\documentclass[12pt,a4paper]{amsart}

\usepackage[top=35mm, bottom=35mm, left=30mm, right=30mm]{geometry}

\usepackage{mathptmx}%字体
\usepackage{mathrsfs}

\usepackage{verbatim}
\usepackage{url}
\usepackage[all]{xy}
\usepackage{color}

\usepackage[colorlinks=true,citecolor=blue]{hyperref}%引用和定理颜色

\usepackage{stmaryrd}
\usepackage{epsfig}
\usepackage{amsmath}
\usepackage{amssymb}

\usepackage{amscd}
\usepackage{graphicx}
\usepackage{pstricks}

\theoremstyle{plain}
\newtheorem{thm}{Theorem}[section]
\newtheorem{lem}[thm]{Lemma}

\newtheorem{prop}[thm]{Proposition}
\newtheorem{ques}{Question}

\theoremstyle{definition}
\newtheorem{de}[thm]{Definition}
\newtheorem{exam}[thm]{Example}
\theoremstyle{remark}
\newtheorem{rem}[thm]{Remark}
\newtheorem{conj}[thm]{Conjecture}

\def \N {\mathbb N}

\def \Z {\mathbb Z}
\def \R {\mathbb R}

\def \B {\mathcal B}

\def \A {\mathcal A}
\def \F {\mathcal F}
\def \G {\mathcal{G}}

\def \X {\mathcal{X}}

\def \O {\mathcal{O}}

\def \E {\mathbb{E}}
\def \T {\mathbb{T}}
\def \ZZ {\mathcal{Z}}

\def \Q {{\bf Q}}
\def \RP {{\bf RP}}
\def \AP {{\bf AP}}

\def \id {{\rm id}}

\def \Ind {{\bf {\rm Ind}}}

\def \ep {\epsilon}
\def \d {\delta}

\def \lra{\longrightarrow}

\begin{document}

\title{Regionally proximal relation of order $d$ along arithmetic progressions and nilsystems}
\author{Eli Glasner}
\author{Wen Huang}
\author{Song Shao}
\author{Xiangdong Ye\\ (Dedicated To Professor Shantao Liao)}

\address{School of Mathematics, Tel Aviv University, Israel}
\email{eli.glasner@gmail.com}

\address{Wu Wen-Tsun Key Laboratory of Mathematics, USTC, Chinese Academy of Sciences and
School of Mathematics, University of Science and Technology of China,
Hefei, Anhui, 230026, P.R. China.}

\email{wenh@mail.ustc.edu.cn} \email{songshao@ustc.edu.cn}
\email{yexd@ustc.edu.cn}


\thanks{This research is supported by NNSF of China (11431012, 11971455, 11571335, 11371339).}

%\date{Jan. 25, 2015}

\begin{abstract}
The regionally proximal relation of order $d$ along arithmetic
progressions, namely $\AP^{[d]}$ for $d\in \N$, is introduced and investigated. It turns out that if
$(X,T)$ is a topological dynamical system with $\AP^{[d]}=\Delta$, then each ergodic measure
of $(X,T)$ is isomorphic to a $d$-step pro-nilsystem, and thus $(X,T)$ has zero entropy.

Moreover, it is shown that if $(X,T)$ is a strictly ergodic distal system with the property that the maximal
topological and measurable $d$-step pro-nilsystems are isomorphic,
then $\AP^{[d]}=\RP^{[d]}$ for each $d\in\N$. It follows that for a minimal $\infty$-pro-nilsystem,
$\AP^{[d]}=\RP^{[d]}$ for each $d\in\N$. An example which is a strictly ergodic distal system
with discrete spectrum whose maximal equicontinuous factor
is not isomorphic to the Kronecker factor is constructed.
%It is an open question if the smallest closed invariant equivalence generated by $\AP^{[d]}$
%is equal to the regionally proximal relation of order $d$, $\RP^{[d]}$.
\end{abstract}

\maketitle

\markboth{Regionally proximal relation of order $d$ along a.p.}
{E. Glasner, W. Huang, S. Shao and X. Ye}

\section{Introduction}

This paper is dedicated to the counterpart of the study of multiple ergodic averages in ergodic theory  in the setting of topological dynamics. The regionally proximal relation of order $d$ along arithmetic
progressions, namely $\AP^{[d]}$ for $d\in \N$, is introduced and investigated.

\medskip

In some sense an equicontinuous system is the simplest system in
topological dynamics. In the study of topological dynamics, one of
the first problems was to characterize the equicontinuous structure
relation $S_{eq}(X)$ of a system $(X, T)$, i.e. to find the smallest
closed invariant equivalence relation $R(X)$ on $(X, T)$ such that
$(X/ R(X), T)$ is equicontinuous. A natural candidate for $R(X)$ is
the so-called regionally proximal relation $\RP(X)$ introduced by Ellis and Gottschalk \cite{EG}.
%Recall that for a system $(X,T)$ and $x,y\in X$,  $(x,y)\in \RP(X)$ if there are sequences
%$x_i,y_i\in X, n_i\in \Z$ such that $x_i\to x, y_i\to y$ and
%$(T\times T)^{n_i}(x_i,y_i)\to (z,z)$, $i\to \infty$, for some $z\inX$.
By the definition, $\RP(X)$ is closed, invariant, and reflexive, but
not necessarily transitive. The problem was then to find conditions
under which $\RP(X)$ is an equivalence relation. It turns out to be
a difficult problem. Starting with Veech \cite{V68}, various
authors, including MacMahon \cite{Mc}, Ellis-Keynes \cite{EK}, Bronstein \cite{Br} etc., came
up with various sufficient conditions for $\RP(X)$ to be an
equivalence relation. Note that in our case, $T: X\rightarrow X$ being
homeomorphism and $(X,T)$ being minimal, $\RP(X)$ is always an
equivalence relation. Using the relative version of equicontinuity, Furstenberg \cite{F63} gave the
structure theorem of a minimal distal system, which had a very important influence both in topological dynamics and ergodic theory.

\medskip

The connection between ergodic theory and additive combinatorics was
built in the 1970's with Furstenberg's beautiful proof of
Szemer\'edi's theorem via ergodic theory \cite{F77}. For a measurable system $(X,\X,\mu,T)$,
Furstenberg asked about the convergence (both in the sense of $L^2(\mu)$ and almost surely) of the multiple ergodic averages
\begin{equation}\label{MEA}
 \frac 1 N\sum_{n=0}^{N-1}f_1(T^nx)\ldots
f_d(T^{dn}x),
\end{equation}
where $f_1, \ldots , f_d \in L^\infty(X,\mu)$.
After nearly 30 years' efforts of many researchers, this problem for the case of $L^2$-convergence was
finally solved in \cite{HK05,Z}. % (see also \cite{Z} for another proof).
In their proofs the notion of characteristic factors, introduced by Furstenberg and Weiss, plays a great
role. Loosely speaking, to
understand the multiple ergodic averages $ \frac 1 N\sum_{n=0}^{N-1}f_1(T^nx)\ldots f_d(T^{dn}x)$,
one can replace each function $f_i$ by its conditional expectation
with respect to some $d$-step pro-nilsystem (the $1$-step pro-nilsystem is the Kroneker factor). Thus one can reduce the problem to the study of the same average in a nilsystem.
In \cite{HK05}, some very useful tools, such as dynamical parallelepipeds, ergodic uniformity seminorms,
structure theory involving pro-nilsystems for ergodic systems etc., were introduced
and obtained  (For the details we refer to the recent book by Host and Kra \cite{HK18}).

\medskip
In the topological setting, Host, Kra and Maass \cite{HKM} obtained a topological structure theory involving
pro-nilsystems for all minimal distal systems, which can be viewed as an analog of the purely ergodic structure
theory of \cite{HK05} and the refinement of the Furstenberg's structure theorem for minimal distal systems.
In \cite{HKM}, a certain generalization of the regionally proximal
relation, namely $\RP^{[d]}$ (the regionally proximal relation of order $d$),
is introduced and used to produce the maximal pro-nilfactors. Precisely, in \cite{HKM} it is shown that if a system is minimal and
distal then $\RP^{[d]}$ is an equivalence relation and $(X/\RP^{[d]},T)$ is the maximal $d$-step pro-nilfactor of the system.
The maximal pro-nilfactor of order $d$, namely $(X/\RP^{[d]},T)$ can be seen as the characteristic factor of
the minimal system $(X,T)$.
In \cite{SY} Shao and Ye show that all these results in fact hold for arbitrarily minimal systems
of abelian group actions. In a recent paper by Glasner, Gutman and Ye \cite{GGY-16}, the same question is considered for
a general group $G$, and similar results are proved. Applications of the above structure theorems can be found in \cite{HKM-1,HSY1}.

\medskip

Earlier the counterpart
of characteristic factors in topological dynamics was studied by
Glasner \cite{G94} from a different point of view, where the characteristic
factors for the action $T\times T^2\times \ldots \times T^n$ are considered. To be precise, let $(X,T)$ be a topological system and $d\in \N$.
Let $\sigma_d=T\times T^2\times \ldots\times T^d$. $(Y,T)$ is said to be an {\em topological
characteristic factor of order $d$ } if there exists a dense $G_\d$
set $\Omega$ of $X$ such that for each $x\in \Omega$ the orbit
%closure $L=\overline{\O}(\underbrace{(x, \ldots,x)}_{d \
%\text{times}}, \sigma_d))$ is $\underbrace{\pi\times \ldots \times
%\pi}_{d \ \text{times}}$ saturated, where $\pi: X\rightarrow Y$ is
closure $L=\overline{\O}({x^d}, \sigma_d)$ is ${\pi\times \ldots \times
\pi}$ ($d$ times) saturated, where $x^d=(x,\ldots,x)$ ($d$ times) and $\pi: X\rightarrow Y$ is
the corresponding factor map. That is, $(x_1,x_2,\ldots, x_d)\in L$
iff $(x_1',x_2',\ldots, x_d')\in L$ whenever for all $1\le i\le d$,
$\pi(x_i)=\pi(x_i')$.
%The notation of topological characteristic factor was suggested by Furstenberg and studied in \cite{G94}.
In \cite{G94}, it is shown that if $(X,T)$ is a distal minimal system, then its largest class $d$ distal
factor (in the Furstenberg's tower of a minimal distal system) is a topological characteristic factor of order $d$; if $(X,T)$ is a weakly mixing system,
then the trivial system is its topological characteristic factor.

\medskip
It is a long open question whether for a minimal distal system in Glasner's theorem in \cite{G94} one can replace the largest class $d$ distal factor by the maximal pro-nilfactor of order $d$.
Indeed, this is the case when we consider characteristic factors along cubes of minimal systems.
In \cite{CS}, the topological characteristic factors along cubes of minimal systems are studied.
It is shown that up to proximal extensions the pro-nilfactors are the topological characteristic factors
along cubes of minimal systems. In particular, for a distal minimal system,
the maximal $(d-1)$-step pro-nilfactor is the topological cubic
characteristic factor of order $d$ \cite{CS}.

\medskip
In this paper, we try to give another way to study the counterpart
of characteristic factors in topological dynamics.
Note that for a minimal system, the maximal pro-nilfactor of order $d$ is obtained by the
regionally proximal relation of order $d$, i.e. $\RP^{[d]}$. Here we propose a direct approach,
that is, we consider the regionally proximal relation of order $d$ along arithmetic
progressions, namely $\AP^{[d]}$ for $d\in \N$.
%We will do this in the current paper.
%The regionally proximal relation of order $d$ along arithmetic
%progressions, namely $\AP^{[d]}$ for $d\in \N$, is investigated.

\medskip
It turns out that if
$(X,T)$ is a topological dynamical system with $\AP^{[d]}=\Delta$, then each ergodic measure
of $(X,T)$ is isomorphic to a $d$-step pro-nilsystem, and thus $(X,T)$ has zero entropy.
We also show that if $(X,T)$ is a strictly ergodic distal system with the property that the maximal
topological and measurable $d$-step pro-nilsystems are isomorphic,
then $\AP^{[d]}=\RP^{[d]}$ for each $d\in\N$. It then follows that for a minimal $\infty$-pro-nilsystem,
$\AP^{[d]}=\RP^{[d]}$ for each $d\in\N$.
We construct an example $(X,T)$ which is a uniquely ergodic minimal distal system
with discrete spectrum whose maximal equicontinuous factor
is not isomorphic to the Kronecker factor.

\medskip
To finish the introduction we make the following
\begin{conj}
Let $(X,T)$ be a minimal distal system. Then $\AP^{[d]}=\RP^{[d]}$ for any $d\in\N.$
\end{conj}
%for any minimal distal system.
Unfortunately, we can not achieve this currently.

\section{preliminaries}

\subsection{Topological dynamical systems}

A {\em transformation} of a compact metric space X is a
homeomorphism of X to itself. A {\em topological dynamical system}
(t.d.s.) or just a {\em system}, is a pair $(X,T)$, where $X$ is a
compact metric space and $T : X \rightarrow  X$ is a transformation.
We use $\rho (\cdot, \cdot)$ to denote a compatible metric in $X$. In the
sequel, and if there is no room for confusion, {\em in any t.d.s. we will always
use $T$ to indicate the transformation}.

\medskip

%We will also make use of a more general
%definition of a system. That is, instead
%of just considering a single transformation $T$, we will consider commuting
%homeomorphisms $T_1,\ldots , T_k$ of $X$. We recall some basic
%definitions and properties of systems in the classical setting of
%one transformation. Extensions to the general case are
%straightforward.

%\medskip

A system $(X, T)$ is {\em transitive} if there exists
$x\in X$ whose orbit $\O(x,T)=\{T^nx: n\in \Z\}$ is dense in $X$ and
such point is called a {\em transitive point}. The system is {\em
minimal} if the orbit of every point is dense in $X$. This is equivalent to saying that X and the empty set are the only closed invariant subsets of $X$.

\medskip

Let $(X,T)$ be a system and let $\B(X)$ be the Borel $\sigma$-algebra. Let $\mathcal{M}(X)$ be the set of Borel probability measures in $X$. A measure $\mu\in\mathcal{M}(X)$ is
{\em $T$-invariant} if for every Borel set $B$ of $X$,
$\mu(T^{-1}B)=\mu(B)$. Denote by $\mathcal{M}(X,T)$ the set of
invariant probability measures. A measure $\mu\in\mathcal{M}(X,T)$
is {\em ergodic} if for any Borel set $B$ of $X$ satisfying
$\mu(T^{-1}B\triangle B)=0$ we have $\mu(B)=0$ or $\mu(B)=1$. Denote
by $\mathcal{M}^e(X,T)$ the set of ergodic measures. The system
$(X,T)$ is  {\em uniquely ergodic} if $\mathcal{M}(X,T)$ consists of
only one element, and it is {\em strictly ergodic} if in addition it is minimal.

\medskip

A {\em homomorphism} between the t.d.s. $(X,T)$ and $(Y,T)$ is a
continuous onto map $\pi: X\rightarrow Y$ which intertwines the
actions; one says that $(Y,T)$ is a {\it factor} of $(X,T)$ and that
$(X,T)$ is an {\em extension} of $(Y,T)$. One also refers to $\pi$
as a {\em factor map} or an {\em extension} and one uses the
notation $\pi: (X,T) \rightarrow (Y,T)$. The systems are said to be
{\em conjugate} if $\pi$ is a bijection. An extension $\pi$ is
determined by the corresponding closed invariant equivalence
relation
$$R_{\pi} = \{ (x_{1},x_2): \pi (x_1)= \pi (x_2) \} =(\pi
\times \pi )^{-1} \Delta_Y \subset  X \times X,$$
where $\Delta_Y$ is
the diagonal on $Y$.
%An extension $\pi: (X,T) \rightarrow (Y,T)$ is
%{\em almost one-to-one} if the $G_\delta$ set $X_0=\{x \in X:
%\pi^{-1}(\pi(x)) = \{x\}\}$ is dense.

\subsection{Distality and Proximality}

Let $(X,T)$ be a t.d.s. A pair $(x,y)\in X\times X$ is a {\it proximal}
pair if
\begin{equation*}
\inf_{n\in \Z} \rho (T^nx,T^ny)=0
\end{equation*}
and is a {\em distal} pair if it is not proximal. Denote by
${\bf P}(X,T)$ or ${\bf P}_X$ the set of proximal pairs of $(X,T)$. The t.d.s.
$(X,T)$ is {\em distal} if $(x,y)$ is a distal pair whenever
$x,y\in X$ are distinct.

\medskip

An extension $\pi : (X,T) \rightarrow (Y,T)$ is  {\em
proximal} if $R_{\pi} \subset {\bf P}(X,T)$ and is {\em distal} if
$R_\pi\cap {\bf P}(X,T)=\Delta_X$. Observe that when $Y$ is trivial (reduced to one point) the map
$\pi$ is distal if and only if $(X,T)$ is distal.

\subsection{Independence}

The notion of \emph{independence} was first introduced and studied in
\cite[Definition 2.1]{KL}. It corresponds to a modification of the notion of
\emph{interpolator} studied in \cite{GW, HY06} and was discussed in depth in
\cite{HLY}.

\begin{de}
Let $(X, T)$ be a t.d.s. Given a tuple $\A=(A_1,\ldots,A_k)$ of
subsets of $X$ we say that a subset $F\subset\Z$ is an {\em
independence set} for $\A$ if for any nonempty finite subset
$J\subset F$ and any $s=(s(j): j\in J ) \in\{1,\ldots,k\}^J$ we have
$$\bigcap\limits_{j\in J}T^{-j}A_{s(j)}\neq\emptyset \ .$$

We shall denote the collection of all independence sets for $\A$ by
$\Ind(A_1,\ldots,A_k)$ or $\Ind\A$.

\end{de}

%A finite subset $F$ of $\Z_{+}$ is called a {\em finite IP-set}  if
%there exists a finite subset $\{p_1,p_2,\ldots,p_m\}$ of $\N$ such
%that
%$$F=FS(\{p_i\}_{i=1}^m)=\{p_{i_1}+\cdots+p_{i_k}:1\leq i_1<\cdots<i_k\leq m\}.$$
%Now we define ${\rm Ind}_{fip}$-pairs.

%\begin{de}
%Let $(X, T)$ be a t.d.s. A pair $(x_1,x_2) \in X \times X$ is called
%an {\em ${\rm Ind}_{fip}$-pair} if for any neighborhoods $U_1$,
%$U_2$ of $x_1$ and $x_2$ respectively, $\Ind(U_1,U_2)$ contains
%arbitrarily long finite IP-sets. Denote by ${\rm Ind}_{fip}(X,T)$
%the set of all ${\rm Ind}_{fip}$-pairs of $(X,T)$.
%\end{de}

%\begin{de}
%Let $(X, T)$ be a t.d.s. and $\F$ be a collection of integers.  A pair $(x_1,x_2) \in X %\times X$ is called
%an {\em ${\rm Ind}_{\F}$-pair} if for any neighborhoods $U_1$,
%$U_2$ of $x_1$ and $x_2$ respectively, $\Ind(U_1,U_2)$ contains
%some element of $\F$. Denote by ${\rm Ind}_{\F}(X,T)$
%the set of all ${\rm Ind}_{\F}$-pairs of $(X,T)$.
%\end{de}

\begin{de}
Let $(X, T)$ be a t.d.s. A pair $(x_1,x_2) \in X \times X$ is called
an {\em ${\rm Ind}_{ap}$-pair} (ap for arithmetic progression) if for every pair of neighborhoods $U_1$,
$U_2$ of $x_1$ and $x_2$ respectively, and every $d\in \N$ there is some $n\in \N$ such that
for each $(t_1,\ldots, t_d)\in \{1,2\}^d$,
$$T^{-n}U_{t_1}\cap T^{-2n}U_{t_2}\cap \ldots\cap
T^{-nd}U_{t_d}\not=\emptyset.$$

Denote by ${\rm Ind}_{ap}(X,T)$ or $\Ind_{ap}(X)$
the set of all ${\rm Ind}_{ap}$-pairs of $(X,T)$.
\end{de}

\subsection{Dynamical parallelepipeds}

 Let $X$ be a set, and let $d\ge 1$ be an integer. We view element in $\{0, 1\}^d$ as a sequence
$\ep=\ep_1\ldots \ep_d$ of $0'$s and $1'$s.
We denote $X^{2^d}$ by $X^{[d]}$. A point ${\bf x}\in X^{[d]}$ can
be written as
${\bf x} = (x_\ep :\ep\in \{0,1\}^d ). $

\begin{de}
Let $(X, T)$ be a topological dynamical system and let $d\ge 1$ be
an integer. We define $\Q^{[d]}(X)$ to be the closure in $X^{[d]}$
of elements of the form $$(T^{{\bf n}\cdot \ep}x=T^{n_1\ep_1+\ldots
+ n_d\ep_d}x: \ep= \ep_1\ldots\ep_d \in\{0,1\}^d) ,$$ where $ x\in X$ and ${\bf
n} = (n_1,\ldots , n_d)\in \Z^d$. When there is no
ambiguity, we write $\Q^{[d]}$ instead of $\Q^{[d]}(X)$. An element
of $\Q^{[d]}(X)$ is called a (dynamical) {\em parallelepiped of
dimension $d$}.
\end{de}

As examples, $\Q^{[2]}$ is the closure in $X^{[2]}=X^4$ of the set
$$\{(x, T^mx, T^nx, T^{n+m}x) : x \in X, m, n \in \Z\}$$ and $\Q^{[3]}$
is the closure in $X^{[3]}=X^8$ of the set $$\{(x, T^mx, T^nx,
T^{m+n}x, T^px, T^{m+p}x, T^{n+p}x, T^{m+n+p}x) : x\in X, m, n, p\in
\Z\}.$$

%\begin{de}
%Let $\phi: X\rightarrow Y$ and $d\in \N$. Define $\phi^{[d]}:
%X^{[d]}\rightarrow Y^{[d]}$ by $(\phi^{[d]}{\bf x})_\ep=\phi x_\ep$
%for every ${\bf x}\in X^{[d]}$ and every $\ep\subset [d]$.
%Let $(X, T)$ be a system and $d\ge 1$ be an integer. The {\em
%diagonal transformation} of $X^{[d]}$ is the map $T^{[d]}$.
%\end{de}

Let $(X, T)$ be a system and $d\ge 1$ be an integer. The {\em
diagonal transformation} of $X^{[d]}$ is the map $T^{[d]}:
X^{[d]}\rightarrow X^{[d]}$ defined by $(T^{[d]}{\bf x})_\ep=T x_\ep$
for every ${\bf x}\in X^{[d]}$ and every $\ep\in \{0,1\}^d$.

\begin{de}
{\em Face transformations} are defined inductively as follows: Let
$T^{[0]}=T$, $T^{[1]}_1=\id \times T$. If
$\{T^{[d-1]}_j\}_{j=1}^{d-1}$ is defined already, then set
\begin{equation}\label{def-T[d]}
\begin{split}
T^{[d]}_j&=T^{[d-1]}_j\times T^{[d-1]}_j, \ j\in \{1,2,\ldots, d-1\},\\
T^{[d]}_d&=\id ^{[d-1]}\times T^{[d-1]}.
\end{split}
\end{equation}
\end{de}

%It is easy to see that for $j\in [d]$, the face transformation
%$T^{[d]}_j : X^{[d]}\rightarrow X^{[d]}$ can be defined by, for
%every ${\bf x} \in X^{[d]}$ and $\ep\subset [d]$,$$T^{[d]}_j{\bf x}=
%\left\{
 % \begin{array}{ll}
  %  (T^{[d]}_j{\bf x})_\ep=Tx_\ep, & \hbox{$j\in \ep$;} \\
  %  (T^{[d]}_j{\bf x})_\ep=x_\ep, & \hbox{$j\not \in \ep$.}
  %\end{array}
%\right.$$

The {\em face group} of dimension $d$ is the group $\F^{[d]}(X)$ of
transformations of $X^{[d]}$ spanned by the face transformations.
The {\em parallelepiped group} of dimension $d$ is the group
$\G^{[d]}(X)$ spanned by the diagonal transformation $T^{[d]}$ and the face
transformations $\F^{[d]}(X)$. We often write $\F^{[d]}$ and $\G^{[d]}$ instead of
$\F^{[d]}(X)$ and $\G^{[d]}(X)$, respectively. For $\G^{[d]}$ and
$\F^{[d]}$, we use similar notations to that used for $X^{[d]}$:
namely, an element of either of these groups is written as $S =
(S_\ep : \ep\in\{0,1\}^d)$. In particular, $\F^{[d]} =\{S\in
\G^{[d]}: S_\emptyset ={\rm id}\}$.

\medskip

For convenience, we denote the orbit closure of ${\bf x}\in X^{[d]}$
under $\F^{[d]}$ by $\overline{\F^{[d]}}({\bf x})$, instead of
$\overline{\O({\bf x}, \F^{[d]})}$.
It is easy to verify that $\Q^{[d]}$ is the closure in $X^{[d]}$ of
$$\{Sx^{[d]} : S\in \F^{[d]}, x\in X\}.$$
If $x$ is a transitive point of $X$, then $\Q^{[d]}$ is the orbit closure of $x^{[d]}$ under the group $\G^{[d]}$.

\subsection{Nilmanifolds and nilsystems}

Let $G$ be a group. For $g, h\in G$ and $A,B \subset G$, we write $[g, h] =
ghg^{-1}h^{-1}$ for the commutator of $g$ and $h$ and
$[A,B]$ for the subgroup spanned by $\{[a, b] : a \in A, b\in B\}$.
The commutator subgroups $G_j$, $j\ge 1$, are defined inductively by
setting $G_1 = G$ and $G_{j+1} = [G_j ,G]$. Let $d \ge 1$ be an
integer. We say that $G$ is {\em $d$-step nilpotent} if $G_{d+1}$ is
the trivial subgroup.

\medskip

Let $G$ be a $d$-step nilpotent Lie group and $\Gamma$ be a discrete
cocompact subgroup of $G$. The compact manifold $X = G/\Gamma$ is
called a {\em $d$-step nilmanifold}. The group $G$ acts on $X$ by
left translations and we write this action as $(g, x)\mapsto gx$.
The Haar measure $\mu$ of $X$ is the unique probability measure on
$X$ invariant under this action. Fix $\tau\in G$ and $T$ be the
transformation $x\mapsto \tau x$ of $X$. Then $(X, \mu, T)$ is
called a {\em $d$-step nilsystem}. In the topological setting we omit the measure
and just say that $(X,T)$ is a $d$-step nilsystem. For more details on nilsystems, refer to \cite{HK18}.

\medskip

We will need to use inverse limits of nilsystems, so we recall the
definition of a sequential inverse limit of systems. If
$(X_i,T_i)_{i\in \N}$ are systems with $diam(X_i)\le 1$ and
$\pi_i: X_{i+1}\rightarrow X_i$ are factor maps, the {\em inverse
limit} of these systems is defined to be the compact subset of
$\prod_{i\in \N}X_i$ given by $\{ (x_i)_{i\in \N }: \pi_i(x_{i+1}) =
x_i\}$, and we denote it by
$\lim\limits_{\longleftarrow}(X_i,T_i)_{i\in\N}$. It is a compact
metric space endowed with the distance $\rho((x_{i})_{i\in\N}, (y_{i})_{i\in
\N}) = \sum_{i\in \N} 1/2^i \rho_i(x_i, y_i )$, where $\rho_{i}$ is the metric in
$X_{i}$. We note that the
maps $T_i$ induce naturally a transformation $T$ on the inverse
limit.

\medskip

The following structure theorem characterizes inverse limits of
nilsystems using dynamical parallelepipeds.

\begin{thm}[Host-Kra-Maass]\cite[Theorem 1.2]{HKM}\label{HKM}
Assume that $(X, T)$ is a transitive topological dynamical system
and let $d \ge 2$ be an integer. The following properties are
equivalent:
\begin{enumerate}
  \item If ${\bf x}, {\bf y} \in \Q^{[d]}$ have $2^d-1$ coordinates in common, then ${\bf x} = {\bf y}$.
  \item If $x, y \in X$ are such that $(x, y,\ldots , y) \in  \Q^{[d]}$,
then $x = y$.
  \item $X$ is an inverse limit of $(d-1)$-step minimal
nilsystems.
\end{enumerate}
\end{thm}

A transitive system satisfying one of the equivalent properties
above is called  a {\it $(d-1)$-step pro-nilsystem} or {\it system of order $(d-1)$}.

\subsection{Regionally proximal relation
of order $d$}

\begin{de}
Let $(X, T)$ be a system and let $d\in \N$. The points $x, y \in X$ are
said to be {\em regionally proximal of order $d$} if for any $\d  >
0$, there exist $x', y'\in X$ and a vector ${\bf n} = (n_1,\ldots ,
n_d)\in\Z^d$ such that $\rho (x, x') < \d, \rho (y, y') <\d$, and $$
\rho (T^{{\bf n}\cdot \ep}x', T^{{\bf n}\cdot \ep}y') < \d\
\text{for every nonempty $\ep\subset [d]$}.$$ In other words, there
exists $S\in \F^{[d]}$ such that $\rho (S_\ep x', S_\ep y') <\d$ for
every $\ep\neq \emptyset$. The set of regionally proximal pairs of
order $d$ is denoted by $\RP^{[d]}$ (or by $\RP^{[d]}(X,T)$ in case of
ambiguity), and is called {\em the regionally proximal relation of
order $d$}.
\end{de}

It is easy to see that $\RP^{[d]}$ is a closed and invariant
relation. Observe that
\begin{equation*}
    {\bf P}(X,T)\subseteq  \ldots \subseteq \RP^{[d+1]}\subseteq
    \RP^{[d]}\subseteq \ldots \RP^{[2]}\subseteq \RP^{[1]}=\RP(X,T).
\end{equation*}
%\medskip

The following theorems proved in \cite{HKM} (for minimal distal systems) and
in \cite{SY} (for general minimal systems) tell us conditions under which
$(x,y)$ belongs to $\RP^{[d]}$ and the relation between $\RP^{[d]}$ and
$d$-step pro-nilsystems, which are defined in Theorem \ref{HKM}.

\begin{thm}\label{thm-1}
Let $(X, T)$ be a minimal system and let $d\in \N$. Then
\begin{enumerate}
\item $(x,y)\in \RP^{[d]}$ if and only if $(x,y,\ldots,y)\in \Q^{[d+1]}$
if and only if $(x,y,\ldots,y) \in
\overline{\F^{[d+1]}}(x^{[d+1]})$.

\item $\RP^{[d]}$ is an equivalence relation.

\item $(X,T)$ is a $d$-step pro-nilsystem if and only if $\RP^{[d]}=\Delta_X$.
\end{enumerate}
\end{thm}

\subsection{$\infty$-step pro-nilsystems}
The regionally proximal relation of order $d$ allows to construct the maximal $d$-step
pro-nilfactor of a system.  In \cite{HKM}, it was shown that for a minimal distal system $(X,T)$ the quotient of $X$ under $\RP^{[d]}(X,T)$ is the
maximal $d$-step pro-nilfactor of $X$. In general one has the following:

%That is, any factor of order $d$ (inverse limit of $d$-step minimal nilsystems) factorize through this system.

\begin{thm}\label{thm0}\cite{SY}
Let $\pi: (X,T)\rightarrow (Y,T)$ be a factor map between minimal systems
and let $d\in \N$. Then,
\begin{enumerate}
  \item $\pi\times \pi (\RP^{[d]}(X,T))=\RP^{[d]}(Y,T)$.
  \item $(Y,T)$ is a $d$-step pro-nilsystem if and only if $\RP^{[d]}(X,T)\subset R_\pi$.
\end{enumerate}
In particular, $X_d=X/\RP^{[d]}(X,T)$, the quotient of $(X,T)$ under $\RP^{[d]}(X,T)$, is the
maximal $d$-step pro-nilfactor of $X$. %(i.e. the maximal factor of order $d$).
\end{thm}

It follows that for any minimal system $(X,T)$,
$$\RP^{[\infty]}=\bigcap\limits_{d\ge 1} \RP^{[d]}$$
is a closed invariant equivalence relation (we write $\RP^{[\infty]}(X,T)$ in case of ambiguity). Now we formulate the
definition of $\infty$-step pro-nilsystems. % or systems of order $\infty$.

\begin{de}\cite{D-Y}
A minimal system $(X, T)$ is an {\em $\infty$-step
pro-nilsystem} or {\em a system of order $\infty$}, if the equivalence
relation $\RP^{[\infty]}$ is trivial, i.e. coincides with the
diagonal.
\end{de}

\begin{rem}
Similar to Theorem \ref{thm0}, one can show that the quotient of a
minimal system $(X,T)$ under $\RP^{[\infty]}$ is the maximal
$\infty$-step pro-nilfactor of $(X,T)$.
\end{rem}

%Let $(X,T)$ be a minimal system. It is easy to see that if $(X,T)$
%is an inverse limit of minimal nilsystems, then $(X,T)$ is an
%$\infty$-step pro-nilsystem. Conversely, if $(X,T)$ is a minimal
%$\infty$-step pro-nilsystem, then $\RP^{[\infty]}=\Delta_X$. For any
%integer $d\geq 1$ let $(X_d,T)$ be the quotient of $(X,T)$ under
%$\RP^{[d]}$. Then $(X,T)=\displaystyle
%\lim_{\longleftarrow}(X_d,T)_{d\in \N}$ as
%$\Delta_X=\RP^{[\infty]}=\bigcap\limits_{d\geq 1}\RP^{[d]}$.
%In fact one can show that a minimal system is an $\infty$-step pro-nilsystem if and only if it is
%an inverse limit of minimal nilsystems.

\section{the regionally
proximal relation of order $d$ along arithmetic progressions}

Now we introduce the notion of
the regionally
proximal relation of order $d$ along arithmetic progressions.

%In this section we will give some basic properties of $\AP^{[d]}$.

\subsection{Definition of $\AP^{[d]}(X,T)$}\

\subsubsection{Definition}
In \cite{HKM} $\RP^{[d]}$ was introduced based on
$d$-dimensional parallelepipeds.
%See \cite{SY} for the further study.
Now we define a relation based on Furstenberg's original average.

\begin{de}\label{arithm} Let $(X,T)$ be a t.d.s. and $d\in\N$.
We say that $(x,y)\in X\times X$ is a {\em regionally proximal pair of order
$d$ along arithmetic progressions} if for each $\d>0$ there exist
$x',y'\in X$ and $n\in\Z$ such that $\rho(x, x') < \d,
\rho(y, y') <\d$ and $$\rho(T^{in}(x'),
T^{in}(y'))<\d\ \text{for each}\ 1\le i\le d.$$

The set of all such
pairs is denoted by $\AP^{[d]}(X)$ and is called the {\em regionally
proximal relation of order $d$ along arithmetic progressions}.
\end{de}

For a relation $B$ on $X$ let $\mathbf{R}(B)$ be the smallest closed invariant equivalence generated by $B$.

\begin{rem}% The following facts are easy to be verified.

\begin{enumerate}
\item When $d=1$, $\AP^{[1]}(X)$ is nothing but the regionally
proximal relation $\RP(X)$.

%\item It is easy to check that $\AP^{[d]}(X)$ is a closed $T\times T$
%invariant relation. We do not know if it is an equivalence relation, i.e. does %$\mathcal{A}(\AP^{[d]}(X))=\AP^{[d]}(X)$?

\item  Note that for ${\bf n}=(n,n,\ldots,n)\in \Z^d$, one has that $\{{\bf n}\cdot \ep: \ep=(\ep_1,\ep_2,\ldots,\ep_d)\in \{0,1\}^d\setminus \{\emptyset\}\}=\{n,2n,\ldots,dn\}$. It follows  easily that $\AP^{[d]}(X)\subset \RP^{[d]}(X)$ for each $d\in\N$, and hence $\mathbf{R}(\AP^{[d]}(X))\subset \RP^{[d]}(X)$.

%\item Let $k\in \N$ then $\AP^{[d]}(X,T)=\AP^{[d]}(X,T^k).$
\end{enumerate}
\end{rem}

\begin{lem}
Let $k\in \N$ then $\AP^{[d]}(X,T)=\AP^{[d]}(X,T^k).$
\end{lem}

\begin{proof}
First we note that $\AP^{[d]}(X,T)\supset \AP^{[d]}(X,T^k).$ Now let $(x,y)\in \AP^{[d]}(X,T)$
then each $\d>0$ there exist $x',y'\in X$ and $n\in\Z$ such that $\rho(x, x') < \d$,
$\rho(y, y') <\d$ and $\rho(T^{in}(x'),
T^{in}(y'))<\d_1$ for each $\ 1\le i\le d,$ where $\d_1<\d$ is such that $\rho(z_1,z_2)<\d_1$ implies
$\rho(T^jz_1,T^jz_2)<\d$ for each $1\le j\le dk$. Then we know that $(x,y)\in \AP^{[d]}(X,T^k)$.
\end{proof}

%\subsection{A difference}
\subsubsection{Comparing $\AP^{[d]}$ with $\RP^{[d]}$}

In order to show that $\RP^{[d]}$ is an equivalence relation in \cite{SY} (see also \cite{HKM})
one proves that $(x,y)\in \RP^{[d]}$ if and only if for each neighborhood $U$ of $y$ there is ${\bf n}=(n_1,\ldots,n_{d+1})\in \Z^{d+1}$
such that $T^{{\bf n}\cdot \ep}(x)\in U$ for each $\ep\not=\emptyset$. Since $\AP^{[1]}=\RP^{[1]}$, it
is natural to ask if for $d=1$ $(x,y)\in \AP^{[1]}$ if and only if for each neighborhood $U$ of $y$ there is
$n\in\N$ such that $T^{n}x, T^{2n}x\in U$. Unfortunately this is not the case as the following example shows.

\medskip

Consider $T:\T^2\lra \T^2, (x,y)\mapsto (x+\alpha, x+y)$, where $\alpha$ is irrational. Then
$T^n(x,y)=(x+n\alpha,y+nx+a(n)\alpha)$ with $a(n)=\frac{1}{2}n(n-1)$. It is easy to see that
$$\RP^{[1]}=\{((x,y_1),(x,y_2)):x,y_1,y_2\in \T\}.$$

Let $y\in \T=[0,1)$ and $y\not=0, \frac{1}{3},\frac{2}{3}$. We claim it is not
true that for each neighborhood $U$ of $(0,y)$ there is $n\in \N$ such that
$T^n(0,0), T^{2n}(0,0)\in U$. Assume that this is the case, i.e. for each
$\ep>0$ there is $n\in\N$ such that
$$-\ep<n\alpha \ ({\text{mod}\ 1})<\ep,\ -\ep<2n\alpha\ ({\text{mod}\ 1})<\ep,$$
and
$$ -\ep<y-a(n)\alpha\ ({\text{mod}\ 1})<\ep,\ -\ep<y-a(2n)\alpha\ ({\text{mod}\ 1})<\ep.$$

A simple calculation shows that $3y=0 \ ({\text{mod}\ 1})$, a contradiction.
Thus, we do not have the property for $\AP^{[1]}$ as for $\RP^{[1]}$.

%\medskip

\subsubsection{A question}

It is easy to check that $\AP^{[d]}(X)$ is a closed $T\times T$
invariant relation. We do not know if it is an equivalence relation, i.e.
\begin{ques}
Is it true that for a minimal t.d.s. $\AP^{[d]}(X)$ is an equivalence relation? If not, is this true when $(X,T)$ is also distal?
\end{ques}

\subsection{Systems with $\AP^{[d]}(X)=X\times X$}\

In this subsection we
show that in some cases we have $\AP^{[d]}(X)=X\times X$.
Glasner studied the diagonal action $\sigma_d=T\times
T^2\times \ldots\times T^d$ and showed that

\begin{thm}\cite{G94} \label{g}
Let $(X,T)$ be a minimal weakly mixing t.d.s. Then for each $d\in\N$
there is a dense $G_\delta$ subset $K_d$ of $X$ such that for each
$x\in K_d$, the orbit of $(x,\ldots,x)$ under $\sigma_d$ is dense in
$X^d$.
\end{thm}

Using this result we have
\begin{prop}\label{wmixing} Let $(X,T)$ be a minimal t.d.s. Then the following
statements are equivalent:
\begin{enumerate}
\item Each pair is an $\Ind_{ap}$-pair, i.e. $\Ind_{ap}(X)=X\times X$.

\item $(X,T)$ is weakly mixing.

\item $\AP^{[d]}(X)=X\times X$ for some $d\ge 2$.
\end{enumerate}
\end{prop}
\begin{proof}(1)$\Rightarrow$(2). Assume that $U_0,U_1$ are
two non-empty open subsets of $X$. Then there is $n\in\N$ such that
$$T^{-n}U_0\cap T^{-2n}U_0\cap T^{-3n}U_0\cap
T^{-4n}U_1\not=\emptyset,$$ which implies that $N(U_0,U_0)\cap
N(U_0,U_1)\not=\emptyset$, and hence $(X,T)$ is weakly mixing.
\medskip

%Another direction:

%It is obvious from the definition.

(2)$\Rightarrow$(1). Let $d\ge 1$, $A_d=\{0,1\}^d=\{S_1,\ldots,
S_{2^d}\}$ and $s=S_1S_2\ldots S_{2^d}\in
\{0,1\}^{d2^d}=\{s_1,\ldots,s_{d2^d}\}$. Assume that $U_0,U_1$ are
two non-empty open subsets of $X$. By Theorem  \ref{g} there are $x\in
X$ and $n\in \N$ such that
$$(T^nx,T^{2n}x, \ldots, T^{(d2^d)n}x)\in \prod_{i=1}^{d2^d} U_{s_i}.$$

This implies that for each $(t_1,\ldots, t_d)\in A_d$
$$T^{-n}U_{t_1}\cap T^{-2n}U_{t_2}\cap \ldots\cap
T^{-nd}U_{t_d}\not=\emptyset,$$ that is, each pair is an
$\Ind_{ap}$-pair.

%Assume that $(X,T)$ is weakly mixing. According to \cite{HY1},
%$E=N(U_0,U_1)$ has lower Banach density 1 for any non-empty open
%subsets $U_0$ and $U_1$. Let $d\in \N$. It is easy to see that if
%$E_1, \ldots, E_d$ have lower Banach density 1 then
%$$E\cap \frac{E\cap 2\N}{2}\cap \ldots\cap \frac{E\cap
%d\N}{d}$$ has lower Banach density 1.

(1)$\Rightarrow$(3) is obvious. To show (3)$\Rightarrow$(1) we
observe that $\RP(X)=X\times X$ which implies weak mixing by well known
results (see, for example, \cite{Au88}).
\end{proof}

\begin{rem} (1). Let $(X,T)$ be a weakly mixing t.d.s. If in addition $(X,T)$ is TE\footnote{TE means topologically ergodic, that is, for all non-empty open sets $U,V\subseteq X$, $N(U,V)$ is syndetic.},
then $(X,T)\times (X,T^2)\times \ldots \times (X,T^d)$ is weakly
mixing (and TE), see \cite[Corollary 4.2]{HY02}. Without the
assumption of TE, this is not true in general. %, see \cite{KP}.

\medskip

(2). Let $(X,T)$ be a t.d.s.. If there is a dense $G_\delta$ set
$X_0$ such that for each $x\in X_0$, $(x,x,\ldots,x)$ has a dense
orbit under $\sigma_d$ in $X^d$, then using the method of the
previous theorem we get that $\AP^{[d]}(X)=X\times X$ for each $d\ge
1$.

%If $(X,T)$ is weakly mixing then $\AP^{[2]}=X\times X$. First we
%assume that $(X,T)$ is weakly mixing and TE. Then $(X,T)\times
%(X,T^2)$ is weakly mixing \cite{HY02}. Note that the transitive
%points is dense $G_\delta$. For any $(x,y)\in Trans_{T\times T^2}$
%we have $(x,y)\in \AP^{[2]}$. It is clear that $\AP^{[2]}=X\times
%X$, since $\AP^{[2]}$ is closed.

%\medskip
%Now we assume that $(X,T)$ is weakly mixing and not TE. So $(X,T)$
%is not an $E$-system.
\end{rem}

To show the following property we need a lemma from \cite{F}.

\begin{lem}\cite[Theorem 1.24]{F}\label{fur}
Let $\N=N_1\cup N_2\cup \ldots \cup N_d$ for some $d\in\N$.
Then there is $1\le i\le d$ such that $N_i$ is piece-wise syndetic.
\end{lem}

%Now we are ready to show

\begin{thm} Let $\pi:(X,T)\lra (Y,T)$ be a proximal extension between two t.d.s.
Then $\AP^{[d]}(X)\supset R_\pi=\{(x,y)\in X^2:\pi(x)=\pi(y)\}$ for any $d\in\N$. In particular, if $(X,T)$ is proximal,
then $\AP^{[d]}(X)=X\times X$ for any $d\in\N$.
\end{thm}
\begin{proof} Assume to the contrary that there are $d\in\N$ and a pair $(x_1,x_2)\in
R_\pi$ but $(x_1,x_2)\not\in \AP^{[d]}.$ This implies that there is
$\ep_0>0$ such that if $0<\ep\le \ep_0$ then for each $m\in\N$ there
is $1\le i\le d$ such that $$\rho(T^{im}x_1,T^{im}x_2)\ge \ep.$$ Let
$$E_i=\{m\in\N:\rho(T^{im}x_1,T^{im}x_2)\ge \ep\}, 1\le i\le d.$$ Then
$\N=E_1\cup\ldots\cup E_d$, and then by Lemma \ref{fur} there is $1\le i\le d$
such that $E_i$ is piece-wise syndetic. This implies that
$E_1\supset iE_i$ is piecewise syndetic. Thus the orbit closure of
$(x_1,x_2)$ under $T\times T$ contains a minimal point which is not on the diagonal, a
contradiction, since $\pi$ is proximal.
\end{proof}

To finish the section we ask

\begin{ques} Let $(X,T)$ be a minimal system and assume that $(x,y)$ is proximal.
Is it true that $(x,y)\in \AP^{[d]}$ for each $d\in\N$?
\end{ques}

\subsection{Factors and extensions}

The following property follows directly by the definition.
\begin{prop}
Let $\pi:(X,T)\lra (Y,T)$ be a factor map between two
systems. Then $\pi\times \pi (\AP^{[d]}(X))\subset\AP^{[d]}(Y)$.
\end{prop}

Generally we do not have $\pi\times \pi
(\AP^{[d]}(X))=\AP^{[d]}(Y)$. For example, let $(X_1,T_1)$ and
$(X_2,T_2)$ be two non-trivial proximal t.d.s. with $X_1\cap
X_2=\emptyset$ and $X=X_1\cup X_2$. Assume that $T:X\lra X$ is such
that $T(x)=T_i(x)$ if $x\in X_i$. Then $(X,T)$ has two minimal
points. It is clear that $\AP^{ [d]}(X,T)=X_1\times X_1\cup
X_2\times X_2\not=X\times X$. Let $(Y,S)$ be the t.d.s. obtained by
collapsing the two minimal points. Then $\AP^{[d]}(Y)=Y\times Y$
since $(Y,S)$ is proximal. Choose $(y_1,y_2)\in Y\times Y$ such that
$y_i$ is not minimal and $y_i\in X_i$. It is clear that
$(y_1,y_2)\not \in \AP^{[d]}(X)$.

\begin{ques}
Let $\pi:(X,T)\lra (Y,T)$ be a factor map between two
minimal systems. Is it true that $\pi\times \pi
(\AP^{[d]}(X))=\AP^{[d]}(Y)$?
\end{ques}

\begin{prop} Let $(X,T)$ be the inverse limit of $(X_i,T_i)$ with
bonding maps $\pi_i:X_{i+1}\lra X_i$. If $x=(x_1,x_2,\ldots),
y=(y_1,y_2,\ldots)\in X$ are such that $(x_i,y_i)\in \AP^{[d]}(X_i)$,
then $(x,y)\in\AP^{[d]}(X)$.
\end{prop}
\begin{proof} Let $\ep>0$, and let $U\times V$ be a neighborhood of $(x,y)$. Then there
are $i\in \N$ and a neighborhood $U_i\times V_i$ of $(x_i,y_i)$ such
that $\pi_i^{-1}(U_i)\subset U$ and $\pi_i^{-1}(V_i)\subset V$.
Since $(x_i,y_i)\in \AP^{[d]}(X_i)$ there are $x_i'\in U_i$ and $y_i'\in V_i$
such that there is $n>0$ large with $\rho(T_i^{jn}x_i',
T_i^{jn}y_i')<\ep'$ ($1\le j\le d$, $\ep'>0$) and
$\rho(\pi_{i-1,k}T_i^{jn}x_i',
\pi_{i-1,k}T_i^{jn}y_i')<\frac{\ep}{2}$, where $1\le j\le d$, $1\le
k\le i-2$. Note that $\pi_{i-1,k}:X_i\lra X_k$ is defined by
$\pi_{i-1,k}=\pi_{k}\ldots \pi_{i-1}$. Now choose $i$
with $\displaystyle \sum_{m=i+1}^\infty
\cfrac{{\rm diam}(X_i)}{2^m(1+{\rm diam}(X_i)}<\frac{\ep}{2}$.
%(if such $n$ does
%not exist then $(x_i,y_i)$ is asymptotic.)

Put $x'=(\ldots x_i' \ldots)\in\pi_i^{-1}(U_i)$ and $y'=(\ldots
y_i'\ldots)\in\pi_i^{-1}(V_i)$. Then $$T^{jn}(x')=(\ldots
T_i^{jn}(x_i')\ldots)\ \text{and}\ T^{jn}(y')=(\ldots
T_i^{jn}(y_i')\ldots).$$ Thus we have $\rho(T^{jn}x',T^{jn}y')<\ep$.
\end{proof}

\section{Systems with $\AP^{[d]}=\Delta$}

In this section we discuss the structure of a t.d.s. with
$\AP^{[d]}=\Delta$, and we show that each ergodic measure
of $(X,T)$ is isomorphic to a system of order $d$,
and in particular $(X,T)$ has zero entropy.

\subsection{Metric description}

\subsubsection{}
Let $d\in \N$. A factor $(Z,\ZZ,\nu,T) $ of $X$ is {\em characteristic} for averages
\begin{equation}\label{F}
    \frac 1 N\sum_{n=0}^{N-1}f_1(T^nx)\ldots
f_d(T^{dn}x)
\end{equation}
if the limiting behavior of (\ref{F}) only depends on the
conditional expectation of $f_i$ with respect to $Z$:
\begin{equation*}
    ||\lim_{N\to \infty}\frac{1}{N}\sum_{n=0}^{N-1}(T^nf_1 T^{2n}f_2 \ldots T^{dn}f_d -
    T^n\E(f_1|\mathcal{Z}) T^{2n}\E(f_2|\mathcal{Z}) \ldots T^{dn}\E(f_d|\mathcal{Z}))||_{L^2}=0
\end{equation*}
for any $f_1,\ldots,f_d\in L^{\infty}(X,\X,\mu)$. The system $Z$ is a universal characteristic
factor if it is a characteristic factor of $X$, and a factor of any
other characteristic factor of $X$. The universal characteristic factor of (\ref{F}) always exists \cite{HK05, Z}, and is denoted by $(Z_{d-1},\mathcal{Z}_{d-1},\mu_{d-1},T)$.

\begin{thm}\cite{HK05}\label{HK}
Let $(X,\X,\mu , T)$ be an ergodic system and $d\in \N$. Then the
system $(Z_{d-1}, \ZZ_{d-1}, \mu_{d-1}, T)$ is a (measure theoretic) inverse
limit of $d-1$-step nilsystems.
$(Z_{d-1}, \ZZ_{d-1}, \mu_{d-1}, T)$ is called a {\it system of order $d-1$}.
\end{thm}

%\begin{lem}\cite[Theorem 1.1.]{HK05}\label{HK}
%Let $(X,\mathcal{B},\mu,T)$ be an ergodic m.p.t. and let $k\geq1$ be
%an integer. Let $f_j$, $1\leq j\leq k$ be $k$ bounded measurable
%functions on $X$, then
%$$\lim_{N\rightarrow\infty}\frac{1}{N}\sum_{n=0}^{N-1} f_1(T^n x)
%f_2(T^{2n} x)\cdots f_k(T^{kn} x)$$ exits in $L^2(X)$. Moreover, it
%indeed equals to the limit
%$$\lim_{N\rightarrow\infty}\frac{1}{N}\sum_{n=0}^{N-1} \E(f_1|\mathcal{Z}_{k-1})(T^n x)
%\E(f_2|\mathcal{Z}_{k-1})(T^{2n} x)\cdots
%\E(f_k|\mathcal{Z}_{k-1})(T^{kn} x).$$
%\end{lem}
%\medskip

We also need the following classic result by Furstenberg.
\begin{thm}\cite{F77}\label{F77}
Let $(X,\mathcal{X},\mu,T)$ be a m.p.t. and let $A\in \mathcal{X}$
be a set with positive measure. Then for every integer $k\geq 1$,
$$\liminf_{N\rightarrow\infty}\frac{1}{N}\sum_{n=0}^{N-1}
\mu(A\cap T^{-n}A\cap T^{-2n}A\cap\cdots\cap T^{-kn}A)>0.$$
\end{thm}

\begin{rem}
In fact by Theorem 1.1. in \cite{HK05}, one can replace $\liminf$ in Theorem \ref{F77} by $\lim$, that is,
$$\lim_{N\rightarrow\infty}\frac{1}{N}\sum_{n=0}^{N-1}
\mu(A\cap T^{-n}A\cap T^{-2n}A\cap\cdots\cap T^{-kn}A)>0.$$
\end{rem}

\subsubsection{}

Let $(X, \mathcal{B},\mu,T)$ be an ergodic m.p.t. and $(Z_k,
\mathcal{Z}_k,\mu_k,T)$ be the $k$-step nilfactor of
$(X,\mathcal{B},\mu,T)$. Let $\displaystyle \mu=\int_{Z_{k}} \mu_z d\mu_k(z)$ be
the disintegration of $\mu$ over $\mu_k$. Pairs in the support of
the measure
$$\lambda_k=\int_{{Z}_{k}} \mu_z\times \mu_z d\mu_k(z)$$
are called $L_k^\mu$-pairs, where $L_k^\mu={\rm Supp}(\lambda_k)$.

Now we may obtain the following theorem related to $L_k^\mu$.

\begin{thm}\label{fibre-ap}
Let $(X,T)$ be a t.d.s. and $\mu$ an ergodic Borel measure on $X$.
Let $k\geq 1$ be an integer, then $L_k^\mu\subset \AP^{[k]}(X)$.
Moreover, $\bigcap_{k=1}^\infty L_k^\mu\subset \Ind_{ap}(X)$.
\end{thm}
\begin{proof}
Let $(x_0,x_1)\in L_k^\mu$. Then for any neighborhood $U_0\times
U_1$ of $(x_0,x_1)$
$$\lambda_k(U_0\times U_1)=\int_{Z_k} \E(1_{U_0}|\mathcal{Z}_k)
\E(1_{U_1}|\mathcal{Z}_k)d \mu_k>0.$$

By Theorem \ref{HK}, we have
\begin{eqnarray*}
&&\lim_{N\rightarrow\infty}\frac{1}{N}\sum_{n=0}^{N-1}
\mu(U_0\cap T^{-n}U_1\cap T^{-2n}U_1\cap\cdots\cap T^{-(k+1)n}U_1)\\
&=& \lim_{N\rightarrow\infty}\frac{1}{N}\sum_{n=0}^{N-1} \int_X
1_{U_0}(x)1_{U_1}(T^n x)1_{U_1}(T^{2n}x)\cdots1_{U_1}(T^{(k+1)n}x)
d\mu(x)\\
&=& \lim_{N\rightarrow\infty}\frac{1}{N}\sum_{n=0}^{N-1}
\int_{{Z}_k}
\E(1_{U_0}|\mathcal{Z}_k)(z)\E(1_{U_1}|\mathcal{Z}_k)(T^n z)
\cdots \E(1_{U_1}|\mathcal{Z}_k)(T^{(k+1)n}z)d\mu_k(z)\\
&\geq& \liminf_{N\rightarrow\infty}\frac{1}{N}\sum_{n=0}^{N-1}
a^{k+2}\int_{{Z}_k} 1_{A_a}(z) 1_{A_a}(T^n z)\cdots
1_{A_a}(T^{(k+1)n}z)d\mu_k(z)\\
&=& \liminf_{N\rightarrow\infty}\frac{1}{N}\sum_{n=0}^{N-1}
a^{k+2}\mu(A_a\cap T^{-n}A_a\cap T^{-2n}A_a\cap\cdots\cap
T^{-(k+1)n}A_a),
\end{eqnarray*}
where $a>0$ and $A_a=\{z\in Z_k:\E(1_{U_0}|\mathcal{Z}_k)(z)>a\
\text{and}\ \E(1_{U_1}|\mathcal{Z}_k)(z)>a\}$.

As $\E(1_{U_0}|\mathcal{Z}_k)\le 1$ and $\E(1_{U_1}|\mathcal{Z}_k)\le 1$, one has that
\begin{equation*}
  \begin{split}
  0< b & :=  \int_{Z_k}
\E(1_{U_0}|\mathcal{Z}_k) \E(1_{U_1}|\mathcal{Z}_k)d \mu_k\\
& = \int_{A_a}
\E(1_{U_0}|\mathcal{Z}_k) \E(1_{U_1}|\mathcal{Z}_k)d \mu_k + \int_{Z_k\setminus A_a}
\E(1_{U_0}|\mathcal{Z}_k) \E(1_{U_1}|\mathcal{Z}_k)d \mu_k \\
& \le \mu_k(A_a)+a\mu_k(Z_k\setminus A_a).
  \end{split}
\end{equation*}
Hence there exists $a>0$ such that $\mu_k(A_a)=b-a \mu_k(Z_k\setminus A_a)>0$.  And so
$$\lim_{N\rightarrow\infty}\frac{1}{N}\sum_{n=0}^{N-1} \mu(U_0\cap
T^{-n}U_1\cap T^{-2n}U_1\cap\cdots\cap T^{-(k+1)n}U_1)>0$$ following
Theorem \ref{F77}. In particular, there is some $x\in X$ such that
$x\in U_0$ and $T^{jn} x\in U_1$ for $j=1,2,\ldots,k+1$.

Given $\epsilon>0$, let $U_0\times U_1$ be a neighborhood of
$(x_0,x_1)$ with diameters of $U_0$ and $U_1$ less that $\epsilon$.
By the above discussion, and letting $x_0'=x$ and $x_1'=T^n(x)$, we
get that $T^{jn} x_0', T^{jn} x_1'\in U_1$ for $j=1,2,\ldots,k$.
This implies that $(x_0,x_1)\in \AP^{[k]}(X)$.

\medskip
Now assume that $(x_0,x_1)\in \bigcap_{k=1}^\infty L_k^\mu$ and
$U_0\times U_1$ is a neighborhood of $(x_0,x_1)$. Let
$(i_0,i_1,\ldots,i_m)\in \{0,1\}^{m+1}$. In a similar discussion as
above, we get that there is $n>0$ such that
$$U_{i_0}\cap T^{-n}U_{i_1}\cap \ldots\cap
T^{-mn}U_{i_m}\not=\emptyset.$$

Since $m$ is arbitrary, for a fixed $k\in\N$ by choosing suitable
$(i_0,i_1,\ldots,i_m)$ we get that $(x_0,x_1)\in \Ind_{ap}(X)$.
To see this, we will show  for every pair of neighborhoods $U_0$,
$U_1$ of $x_0$ and $x_1$ respectively, and every $d\in \N$ there is some $n\in \N$ such that
for each $(t_1,\ldots, t_d)\in \{0,1\}^d$,
$$T^{-n}U_{t_1}\cap T^{-2n}U_{t_2}\cap \ldots\cap
T^{-nd}U_{t_d}\not=\emptyset.$$
Let $\{0,1\}^d=\{t^{(k)}=(t^{(k)}_1, \ldots,t^{(k)}_d): k=1,2,\ldots,2^d\}$. Let $m=d\cdot 2^d$ and $(i_1,\ldots,i_m)=(t^{(1)},t^{(2)},\ldots,t^{(2^d)})\in \{0,1\}^{d\cdot 2^d}$. There is $n>0$ such that
$$T^{-n}U_{i_1}\cap \ldots\cap
T^{-mn}U_{i_m}\not=\emptyset.$$
Set $V_{t^{(k)}}=T^{-n}U_{t^{(k)}_1}\cap T^{-2n}U_{t^{(k)}_2}\cap \ldots\cap
T^{-nd}U_{t^{(k)}_d}$. Then one has that
$$T^{-n}U_{i_1}\cap \ldots\cap
T^{-mn}U_{i_m}=V_{t^{(1)}}\cap T^{-n} V_{t^{(2)}} \cap T^{-2n} V_{t^{(2)}}\cap \ldots \cap T^{-(2^d-1)n} V_{t^{(2^d)}}.$$
It follows that $V_{t^{(k)}}\neq \emptyset$ for all $k\in \{1,2,\ldots,2^d\}$. Since $\{0,1\}^d=\{t^{(k)}: k=1,2,\ldots,2^d\}$, we have that  for each $(t_1,\ldots, t_d)\in \{0,1\}^d$,
$$T^{-n}U_{t_1}\cap T^{-2n}U_{t_2}\cap \ldots\cap
T^{-nd}U_{t_d}\not=\emptyset.$$
The proof is complete.
\end{proof}

A direct application of the above theorem is the following.
\begin{thm}
Let $(X,T)$ be a t.d.s. with $\AP^{[d]}(X)=\Delta$ for some integer
$d\geq 1$, then for each ergodic Borel measure $\mu$, $(X,\mu,T)$ is
measure theoretical isomorphic to a $d$-step pro-nilsystem.

Likewise, if
$\AP^{[\infty]}(X)=\Delta$, then for each ergodic Borel measure $\mu$,
$(X,\mu,T)$ is measure theoretical isomorphic to a $\infty$-step
pro-nilsystem.
\end{thm}

\subsection{Topological description}
%We want to show that if $(X,T)$ is a minimal t.d.s. with $\AP^{[2]}(X)=\Delta$ then $(X,T)$ is PI.

%{\color{red}
%We may use the following theorem to show PI: each transitive system in $X^2$ with dense minimal points is minimal.}
%\begin{prop}\label{proximal} Let $(X,T)$ be a t.d.s. and $(x,y)$ be proximal. If the orbit closure of $(x,y)$ contains only one minimal set,
%then $(x,y)\in \AP^{[d]}(X)$ for any $d\in \N$. Consequently, if $\pi:(X,T)\lra (Y,S)$ is a proximal extension, then
%$R_\pi\subset \AP^{[d]}(X)$ for any $d\in \N$.
%\end{prop}}
%\begin{proof} It is known that if $\pi$ is proximal, then each
%minimal subset of $R_{\pi}$ is contained in the diagonal. Hence by
%Theorem \ref{beijing}, the proposition holds.
%\end{proof}

%{\color{red}
%\begin{prop} Let $\pi:X\lra Y$ be a weakly mixing RIC extension between
%minimal systems, then $R_\pi\subset\AP^{[d]}(X)$ for any $d\in \N$.
%\end{prop}

%\begin{thm} Let $(X,T)$ be a minimal t.d.s. with $\AP^{[2]}(X)=\Delta$ then $(X,T)$ is PI.
%\end{thm}}
%\begin{proof} It $(Y,S)$ is totally minimal, it follows from Eli's proof.
%\end{proof}

%{\color{red}
%\begin{prop} Let $\pi_1: X\lra Y$ be a distal and $\pi_2:Y\lra Z$ be a proximal extension
%between minimal systems. Then $\AP^{[2]}(X)\not=\Delta$. ??
%\end{prop}}
%\begin{proof} By Proposition \ref{proximal}, if $P(X)$ is closed then $\AP^{[2]}(X)\not=\Delta$.
%So now we assume that $P(X)$ is not closed.
%%We know that $R_{\pi_2}(Y)\subset \AP^{[2]}(Y).$ This implies that
%\end{proof}

%The next proposition is interesting.
Let $(X,T)$ be a t.d.s.
A pair $(x,y)\in X \times X$ is said to be {\em asymptotic} when $\lim_{n \to +\infty} d(T^nx,T^ny)=0$.
The set of asymptotic pairs of $(X,T)$ is denoted by $Asym(X,T)$.
%We use $\mathbf{R}(\AP^{[d]}(X))$ to denote the smallest closed invariant equivalence relation containing $\AP^{[d]}(X)$.

The notion of entropy pair was introduced by Blanchard in \cite{B1, B2}. Let $(X,T)$ be a  t.d.s. and $x,x'$ two distinct points of $X$. Call $(x,x')\in X^2$ an {\em entropy pair} of $(X,T)$ if for every open cover $\{U,V\}$ of $X$ with $x\in {\rm int}(U^c), x'\in {\rm int}(V^c)$ we have that the entropy $h_{top}(\{U,V\})>0$.
Let $E(X,T)$ be the set of entropy pairs.

%Let $\Ind_{ap}$ be the collection of pairs of $X$ with arbitrarily long finite arithmetic progressions.
It is proved in \cite{HLY1} that $\Ind_{ap}$ has the lifting property. Since ${\bf P}\subset
\RP^{[d]}$, we know that $X/\RP^{[d]}$ is distal, and hence has zero
entropy. Here we have

\begin{prop}\label{zeroentropy} Let $(X, T)$ be a t.d.s.. Then
\begin{enumerate}
\item $Asym(X,T)\subset\AP^{[d]}(X)$. Consequently, $X/\mathbf{R}(\AP^{[d]}(X))$ has zero topological entropy.

\item $E(X,T)\subset \Ind_{ap}(X,T)\subset \AP^{[d]}(X,T)$. This also implies that
$X/{\mathbf R}(\AP^{[d]}(X))$ has zero topological entropy.

\item  If $\mu\in M^e(X,T)$ is not measure theoretically isomorphic to an $\infty$-step nilsystem,
then $\Ind_{ap}\not=\Delta_X$, where an $\infty$-step nilsystem means that it is an inverse
limit of minimal nilsystems. This also implies that
$X/\mathbf{R}(\AP^{[d]}(X))$ has zero topological entropy.
\end{enumerate}
\end{prop}

\begin{proof}(1) It is easy to see that $Asym (X,T) \subset \AP^{[d]}(X)$.
It follows that $\overline{Asym(X,T)}\subset \AP^{[d]}(X)$. By the
result of \cite{BHR} we know that $X/\mathbf{R}(\AP^{[d]}(X))$ has
zero topological entropy.

\medskip

(2)  It was shown in \cite{HY06} that if $(x_1,x_2)\in E(X,T)$ then each neighborhood $U_1\times U_2$ of $(x_1,x_2)$ has an
independence set of positive density. By the famous Szemer\'{e}di's
theorem, each positive density set contains arbitrarily long
arithmetic progressions, which implies that $E(X,T)\subset
\Ind_{ap}(X,T)\subset \AP^{[d]}(X,T)$. Thus, one gets that
$X/\mathbf{R}(\AP^{[d]}(X))$ has zero topological entropy.

\medskip
(3) The proof is similar to that of Theorem 6.4 of \cite{D-Y}. If $\Ind_{ap}(X)=\Delta_X$,
then by Theorem \ref{fibre-ap} we have that $\bigcap_{k\in \N} L^\mu_k\subset \Ind_{ap}(X)=\Delta_X$.
It is easy to verify that $\bigcap_{k\in \N} L^\mu_k=L^\mu_\infty$ where $L^\mu_\infty$ is the support
of $\lambda_\infty=\int_{{Z}_{\infty}} \mu_z\times \mu_z d\mu_\infty(z)$ with $(Z_\infty,\mu_\infty)$
the inverse limit of $(Z_k,\mu_k)$. Thus for each ergodic measure $\mu$, $(X,\mu,T)$ is measure
theoretically isomorphic to an $\infty$-step nilsystem.

If $\mu\in M^e(X,T)$ is not measure theoretically isomorphic to an $\infty$-step nilsystem,
then $\Ind_{ap}\not=\Delta_X$. Since $\Ind_{ap}(X,T)$ has
the lifting property and $\Ind_{ap}(X,T)\subset
\AP^{[d]}(X,T)$, we conclude that $X/\mathbf{R}(\AP^{[d]}(X))$ has
zero topological entropy.
\end{proof}

\section{For a $d$-step pro-nilsystem $\AP^{[i]}=\RP^{[i]}$, $1\le i\le d$}

In this section we show that for a $d$-step
pro-nilsystem, $$\AP^{[i]}=\RP^{[i]}, 1\le i\le d.$$
Hence at least in this case, $\AP^{[i]}$ is an equivalence relation.

\subsection{$\AP^{[i]}=\RP^{[i]}$ for $i\le d$}

%\section{RIM}

Let $X$ be a compact metric space and let $\mathcal{M}(X)$ be the collection
of regular Borel probability measures on $X$ provided with the weak
star topology. Then $\mathcal{M}(X)$ is a compact metric space in which $X$ is
embedded by the mapping $x\mapsto \d_x$, where $\d_x$ is the dirac
measure at $x$. If $\phi: X\rightarrow Y$ is a continuous map
between compact metric spaces, then $\phi$ induces a continuous map
$\phi^*: \mathcal{M}(X)\rightarrow \mathcal{M}(Y)$ which extends $\phi$, where
$(\phi^*\mu)(A)=\mu(\phi^{-1}A)$ for all Borel sets $A\subseteq Y$.

%Let $(X,T)$ be a t.d.s.. For each $\mu\in M(X)$, define
%$(T\mu)(A)=\mu(T^{-1}A)$ for all Borel sets $A\subseteq X$. Then
%$(M(X), T)$ is a t.d.s. too. And if $\pi: X\rightarrow Y$ is an
%extension of t.d.s., then $\pi^*: M(X)\rightarrow M(Y)$ is also an
%extension.

\begin{de}
An extension $\pi: (X,T)\rightarrow (Y,T)$ of t.d.s. is said to have a {\em
relatively invariant measure} (RIM for short) if there exists a
continuous homomorphism $\lambda:Y\rightarrow \mathcal{M}(X)$ of t.d.s. such
that $\pi^*\circ \lambda : Y\rightarrow \mathcal{M}(Y)$ is just the (dirac)
embedding.
\end{de}
In other words: $\pi$ is a RIM extension if and only if
for every $y\in Y$ there is a $\lambda_y\in \mathcal{M}(X)$ with ${\rm supp}
\lambda_y\subseteq \pi^{-1}(y)$ and the map $y\mapsto \lambda_y:
Y\rightarrow \mathcal{M} (X)$ is a homomorphism of t.d.s; this map $\lambda$ is
called a {\em section} for $\pi$. Note that $\pi: X\rightarrow
\{\star\}$ has a RIM if and only if $X$ has an invariant measure if
and only if $\mathcal{M}(X)$ has a fixed point, where $\{\star\}$ stands for
the trivial system.

%An extension $\pi: X\rightarrow Y$ is called
%{\em strongly proximal} if for every pair $\mu\in M(X)$ and $y\in Y$
%with ${\rm Supp} \mu\subseteq \pi^{-1}(y)$, a sequence $\{n_i\}$ can
%be found such that $\lim T^{n_i}\mu$ is a point mass. It is easy to
%see that each strongly proximal extension is proximal.

%\begin{lem}\cite[Lemma 3.3]{G75}\label{Glasner-lem}
%Let $\pi: X\rightarrow Y$ be a RIM extension of minimal systems and $\lambda: %Y\rightarrow M(X)$ be a
%section for $\pi$. If $X$ is a metric space then there exists a residual set %$Y_0\subseteq Y$ such
%that $y\in Y_0$ implies ${\rm Supp} \lambda _y=\pi^{-1}(y)$.
%\end{lem}

%\begin{rem}\label{ss-1}
%By the proof of \cite[Lemma 3.3]{G75}, we know that if in addition $\pi$ is open, then $Y_0=Y$.
%\end{rem}

\begin{de}\label{Group-extension}
An extension $\phi: (Z,T)\rightarrow (Y,T)$ is called a {\em group extension} with group $G$ if the following conditions are fulfilled:
\begin{enumerate}
  \item $G$ is a compact Hausdorff topological group, acting continuously on $Z$ from the right as a
  group of automorphisms of the system $Z$; this means that there is a continuous mapping $(x,g)\mapsto xg: Z\times G\rightarrow Z$ such that
      \begin{enumerate}
        \item (right action): $\forall x\in Z, \forall g_1,g_2\in G$, $x(g_1g_2)=(xg_1)g_2, x e_G=x$;
        \item (Automorphisms): $\forall g\in G, \forall x\in Z$: $T(xg)=(Tx)g$;
      \end{enumerate}
  \item The fibers of $\phi$ are precisely the $G$-orbits in $Z$, that is, for all $x\in Z$, $\phi^{-1}(\phi(x))=xG$.
\end{enumerate}

\end{de}

A basic theorem about equicontinuous extension is the following result:

\begin{thm}\cite{EGS}\label{equi-extension}
Let $\pi: X\rightarrow Y$ be an extension of minimal systems. Then $\pi$ is equicontinuous if and only
if it is a factor of a group extension, that is, we have the following commutative diagram with $\phi$ a group extension:
%\medskip
$$
\xymatrix{
& X \ar[d]_{\pi}
& {Z}\ar[l]_{\psi} \ar[ld]^{\phi} \\
&Y}
$$

%\medskip
\end{thm}

Glasner showed that every distal extension has a RIM \cite[Propsition 3.8.]{G75}, for our purpose we need a little more.

\begin{prop}\label{Glasner-prop}%\cite[Propsition 3.8.]{G75}
Let $(X,T)$ be a minimal system and let $\pi: (X,T)\rightarrow (Y,T)$ be
a distal extension. Then $\pi$ has a RIM with a section $\lambda$ such that ${\rm Supp} \lambda _y=\pi^{-1}(y)$ for all $y\in Y$.
\end{prop}

\begin{proof}
One can find the proof of the first part of the statements in \cite{G75} or \cite[Chapter V. (6.5)]{Vr}. Since we need
to show the second part of the statements,  we give the whole proof of the results for completeness.

\medskip
Let $\pi: (X,T)\rightarrow (Y,T)$ be a factor between two minimal systems. Then by Furstenberg structure theorem for distal extensions $\pi$ is distal if and only if there exists a countable ordinal $\zeta$
and a directed family of factors $(X_\theta, T), \theta \le \zeta$ such that
\begin{enumerate}
  \item $X_0=Y$ and $X_\zeta=X$.
  \item For $\theta<\zeta$ the extension $\pi_\theta : X_{\theta+1}\rightarrow X_\theta$ is equicontinuous.
  \item For a limit ordinal $\xi \le \zeta$, $X_\xi={\lim\limits_{\longleftarrow}}_{\theta<\xi} X_\theta$.
\end{enumerate}

For convenience if a section satisfies ${\rm Supp} \lambda _y=\pi^{-1}(y)$ for all $y\in Y$ then we say it is a {\em section with full support}.
Hence to prove the result, we need to show (I) each equicontinuous extension has a section with full support; (II) a (transfinite) composition of RIM with full support section has a RIM with full support section.

\medskip

\noindent {\em (I). Each equicontinuous extension has a section $\lambda$ and ${\rm Supp} \lambda _y=\pi^{-1}(y)$ for all $y\in Y$.}

\medskip

Let $\pi: X\rightarrow Y$ be an equicontinuous extension of minimal systems. Now by Theorem \ref{equi-extension} we have the following diagram with $\phi$ a group extension:
\medskip

$$
\xymatrix{
& X \ar[d]_{\pi}
& {Z}\ar[l]_{\psi} \ar[ld]^{\phi} \\
&Y}
$$

\medskip
Now assume that $\phi$ satisfies all conditions in Definition \ref{Group-extension}. For $x\in Z$ and $f\in C(Z)$, let
$$\hat\lambda_{\phi(x)}(f)=\int_K f(xg) d\mu (g),$$
where $\mu$ is the Haar measure on the group $G$.
Then $Y\rightarrow \mathcal{M}(Z), y\mapsto \hat\lambda_y$ is a section for $\phi$.
Since for all $x\in Z$, $\phi^{-1}(\phi(x))=xG$ and the definition of $\hat\lambda$, we have $${\rm Supp} \hat\lambda _{\phi(x)}={\rm Supp} \mu =x G=\phi^{-1}(\phi(x))$$ for all $x\in Z$.

Now let $\lambda=\psi_*\circ \hat\lambda: Y\rightarrow \mathcal{M}(X)$, where $\psi_*:\mathcal{M}(Z)\rightarrow \mathcal{M}(X)$ is the map induced by $\psi: Z\rightarrow X$. Then $\lambda$ is a section for $\pi$, and
$${\rm Supp} \lambda_y=\psi({\rm Supp} \hat\lambda _y)=\psi(\phi^{-1}(y))=\pi^{-1}(y)$$ for all $y\in Y$. This ends the proof for equicontinuous extensions.

\medskip

\noindent {\em (II) A (transfinite) composition of RIM with full support section has a RIM with full support section.}

\medskip
To prove the statement, it suffices to show two cases. One is the composition of two extensions with the properties, the other is the inverse limit of extensions.

First let $\pi_1:X\lra Y$ and $\pi_2:Y\lra Z$ be open factor extensions between minimal systems and let $\pi_1,\pi_2$ have RIM with full support sections.
Let $\lambda^1:Y\lra {\mathcal{M}}(X)$ and $\lambda^2:Z\lra {\mathcal{M}}(Y)$ be two sections.
Define $\eta:Z\lra {\mathcal{M}}(X)$ such that for each $z\in Z$
$$\eta_z(f)=\int_Y(\int_Xf\ \text{d}\lambda^1_y(x))\ \text{d}\lambda^2_z(y),$$
for each $f\in C(X)$. To check that $\eta$ is a section we need to show that $\eta$ is continuous and
$(\pi_2\circ \pi_1)^*(\eta_z)=\delta_z$. The continuity of $\eta$ follows from that of $\lambda^i$, $i=1,2$.
Now we check that $(\pi_2\circ \pi_1)^*(\eta_z)=\delta_z$. In fact
\begin{eqnarray*}
(\pi_2\circ \pi_1)^*(\eta_z)(B)&=&\eta_z(\pi_1^{-1}\circ \pi_2^{-1}(B))\\
&=&\int_Y(\int_X 1_{\pi_2^{-1}(B)} d \delta_y)d\lambda^2_z(y)=\int_Y 1_Bd\delta_z =\delta_z(B),
\end{eqnarray*}
for each $B\in {\mathcal{B}}(Z)$, since $\lambda^i$, $i=1,2$ is a section.

Finally we show $\text{Supp}(\eta_z)=(\pi_2\circ \pi_1)^{-1}(z)$ for each $z\in Z$.
Fix $z\in Z$ and assume that $x\in (\pi_2\circ \pi_1)^{-1}(z)$ and $U$ is an open neighborhood of $x$.
Then
$$\eta_z(U)=\int_Y(\int_X 1_U d\lambda^1_y(x))d\lambda^2_z(y)=\int_Y\lambda^1_y(U)d\lambda^2_z(y) >0,$$
since (1) $\pi_1(U)$ is open in $Y$, and $\pi_2\circ \pi_1(U)$ is open in $Z$, (2) for $y\in \pi_1(U)$,
$\lambda^1_y(U)>0$ and $\lambda^2_z(\pi_1(U))>0$.

\medskip

Next we discuss the inverse limit. Assume that $X$ is an inverse limit of $X_n$.
Let $\pi_n:X\lra X_n$ and $\pi_{n,m}:X_n\lra X_m$ if $n\ge m$ (we set $\pi_{n,n}=id$). For any $x\in X_1$ and $f\in C(X)$ define
$$\eta_x(f)=\lim_{n\lra \infty} (\eta_n)_x(f_n),$$ if $f$ is a limit of $f_n\circ\pi_n$ with $f_n\in C(X_n)$.
Here $(\eta_n)_x\in {\mathcal{M}}(X_n)$ is defined by induction using previous argument.

It is easy to check that $\eta_x$ is well defined. Moreover, if $f=f_n\pi_n$ for some $n\in\N$ then
$\eta_x(f)=(\eta_{n+i})_x(f_n\pi_{n+i,n})=(\eta_n)_x(f_n)$ for $i\ge 0$.

Then we check that $\eta:X_1\lra {\mathcal{M}(X)}$
is a section. To show the continuity of $\eta$, assume that $y_n\lra y$. We will show $\eta_{y_n}\lra \eta_y$, i.e.
for each $f\in C(X), \eta_{y_n}(f)\lra \eta_y(f)$. This follows from the facts that when $f$ is close to $f_k\pi_k$
in $C(X)$, $\eta_{z}(f)$ is close to $\eta_z(f_k\pi_k)=(\eta_k)_z(f_k)$ for each $z\in X_1$ uniformly;
and $\eta_k:X_1\lra {\mathcal{M}}(X_k)$ is continuous.
% follows from that of $\eta_n$, $n\in\N$.

We are left to show that
$(\pi_1)^*\eta_{x_1}(B)=\delta_{x_1}(B)$ for each $B\in {\mathcal{B}}(X_1)$. In fact
$$(\pi_1)^*\eta_{x_1}(B)=\eta_{x_1}(\pi^{-1}_1(B))=\lim_{n\lra \infty} (\eta_n)_{x_1}(\pi_{n,1}^{-1}(B))=\delta_{x_1}(B).$$

To show $\eta$ is full, we note that
$\{\pi_n^{-1}(U_n): U_n\ \text{is open in}\ X_n,\ n\in\N\}$ is a base for the topology of $X$. Fix $x_1\in X_1$
and let $x\in \pi_1^{-1}(x_1)$ and $U$ be an open neighborhood of $x$. Then there is $n\in\N$ such that $U\supset \pi_n^{-1}(U_n)$
and $x_1\in \pi_{n,1}(U_n)$,
where $U_n$ is an open set in $X_n$. Then $\eta_{x_1}(U)\ge (\eta_n)_{x_1}(U_n)>0$. The proof is completed.
\end{proof}

Now we have the following result:

\begin{prop}\label{shaosong}
Let $(X,T)$ be a strictly ergodic system with unique invariant measure $\mu$ and let $\pi: (X,T)\rightarrow (Y,T)$ be a
distal extension. Let $\displaystyle \mu=\int_Y\mu_y\ d\ \nu(y)$ be the disintegration of $\mu$ over $\nu$, where $v=\pi_*(\mu)$.
Then there is $Y_0\subset Y$ with full measure such that
for each $y\in Y_0$, ${\rm Supp}(\mu_y)=\pi^{-1}(y)$.
\end{prop}

\begin{proof}
Since $\pi$ is distal, it has a RIM by Proposition \ref{Glasner-prop}. Let $\lambda: Y\rightarrow M(X)$ be a section for $\pi$.
Then $\widetilde{\mu}=\int\lambda_y\ d\ \nu(y)$ is an invariant measure of $(X,T)$. By unique ergodicity,
$\widetilde{\mu}=\mu$. Since the disintegration is unique, there is $Y_0\subset Y$ with full measure such that
for each $y\in Y_0$, $\mu_y=\lambda_y$. Thus the result follows from Proposition \ref{Glasner-prop}.
\end{proof}

The following lemmas come from \cite{D-Y}.

\begin{lem}\label{mes-top} Let $(X,T)$ be a minimal system of order $n$; then the maximal measurable
and topological factors of order $d$ coincide, where $d\le n$.
\end{lem}

\begin{lem}\label{5p-1} Let $(X,T)$ be a minimal $\infty$-step pro-nilsystem. Then $(X,T)$ is an inverse limit of minimal $d_i$-step
nilsystems. % with $d_i\to \infty$.
\end{lem}

Recall that for $d\in \N$, $X_d=X/\RP^{[d]}$.

\begin{lem}\label{equality}  Let $(X,T)$ be a minimal system. If $X_n=X_{n+1}$ then $X_k=X_n$ for any $k\ge n$.
\end{lem}

%Let $(X,T)$ be an $\infty$-nilsystem. Let $\mu$ be the unique ergodic measure on $(X,T)$
%(see \cite{D-Y}). Let $\pi: X\lra X_{n}=X/\RP^{[n]} $ be the factor map and $\nu=\pi(\mu)$. By Lemma~ \ref{mes-top} $\pi$
%can be viewed as the factor map from $X$ to $Z_{n}$.
%Let $\mu=\int_{Z_{n}} \mu_z d\nu(z)$ be the disintegration of $\mu$ over $\nu$
%and $\lambda=\int_{{X}_{n+1}} \mu_z\times \mu_z d\nu(z).$ Under those assumptions we have

%\begin{lem} \label{f-p} There is $Y_0\subset X_{n}$ with full measure such that
%for each $y\in Y_0$, there is $y\in X_{n}$ such that $supp(\mu_y)=\pi^{-1}(y)$
%and $supp(\mu_y)\times supp(\mu_y)\subset supp(\lambda)$.
%\end{lem}
%\begin{proof}
%\end{proof}

\medskip

Now it is time to give the main result of this section.

\begin{thm} Let $(X,T)$ be a unique ergodic minimal distal system such that for each $d\ge 1$,
$Z_d$ is isomorphic to $X_d$. Then for $d\ge 1$, $\AP^{[d]}=\RP^{[d]}$.

Consequently, for a
minimal $\infty$-nilsystem, we have for $d\ge 1$, $\AP^{[d]}=\RP^{[d]}$.
%Let $(X,T)$ be an $\infty$-nilsystem. Then for each $d\ge 1$, $\AP^{[d]}=\RP^{[d]}$.
\end{thm}
\begin{proof} We use induction. It is clear that $d=1$, $\AP^{[1]}=\RP^{[1]}$. Assume that
$\AP^{[d]}=\RP^{[d]}$ for $1\le d\le n$. Let $\mu$ be the unique ergodic measure on $(X,T)$.
Let $\pi: X\lra X_{n+1}=X/\RP^{[n+1]} $ be the factor map and $\nu=\pi(\mu)$. By the assumption, %Lemma~ \ref{mes-top}
$\pi$ can be viewed as the factor map from $X$ to $Z_{n+1}$.
%If $X_n=X_{n+1}$ we are done by
%Lemma \ref{equality}. So we assume that $X_n\not=X_{n+1}$ and hence $Z_n\not=Z_{n+1}$.

Let $\mu=\int_{X_{n+1}} \mu_z d\nu(z)$ be the disintegration of $\mu$ over $\nu$ and
$$\lambda=\int_{{X}_{n+1}} \mu_z\times \mu_z d\nu(z).$$
By Theorem \ref{fibre-ap}, ${\rm Supp}(\lambda)\subset \AP^{[n+1]}.$

We are going to show that ${\rm Supp}(\lambda)=R_\pi$. First we note that $\lambda(R_\pi)=1$, so
${\rm Supp}(\lambda)\subset R_\pi$. By Proposition \ref{shaosong} there is a measurable set
$Y_0\subset X_{n+1}$ with full measure such that for any $y\in Y_0$, ${\rm Supp}(\mu_y)=\pi^{-1}(y)$. Let $W= {\rm Supp}(\lambda)$. Since
$$\lambda(W)=\int_{Y_0}\mu_y\times \mu_y(W) d\nu(y)=1,$$
we have that for a.e. $y\in Y$, $\mu_y\times \mu_y(W)=1$. This implies that
${\rm Supp}(\mu_y)\times {\rm Supp}(\mu_y)\subset W$, a.e. $y\in Y$.

%For any $(x_1,x_2)\in R_\pi$
%with $y=\pi(x_1)=\pi(x_2)\in Y_0$, let $U_i$ be an open neighborhood of $x_i$ for $i=1,2$. Then
%$U=\pi(U_1)\cap \pi(U_2)$ is an open neighborhood of $y$, since $\pi$ is open. It is clear that $\nu(U)>0$.
%So for each $z\in U$, $\pi^{-1}(z)\cap U_i\not=\emptyset$
%for $i=1,2$ and
%$$\lambda(U_1\times U_2)=\int_{X_{n+1}}\mu_y(U_1)\mu_y(U_2) d\nu(y)\ge \int_{U}\mu_y(U_1)\mu_y(U_2) d\nu(y) >0,$$ which implies that
%$(x_1,x_2)\in {\rm Supp}(\lambda)$.
Thus by the distality of $\pi$, we have ${\rm Supp}(\lambda)=R_\pi$. Thus, %and the fact that ${\rm Supp}(\lambda)\subset \AP^{[n+1]}$, we have
$$R_\pi ={\rm Supp}(\lambda)\subset \AP^{[n+1]}.$$ Since $\AP^{[n+1]}\subset \RP^{[n+1]}$,
we conclude that $\AP^{[n+1]}= \RP^{[n+1]}$. This ends the proof of the first statement of the theorem.

\medskip

When $(X,T)$ is a minimal $\infty$-step pro-nilsystem, $(X,T)$ is uniquely ergodic, see \cite{D-Y}.
The result follows from what we just proved, Lemmas \ref{mes-top}-\ref{equality}
and an inverse limit argument.
\end{proof}

%\begin{ques} Is it true that for a minimal distal system, $\AP^{[d]}=\RP^{[d]}$?
%\end{ques}

%It is clear that there is $y\in X_{n+1}$ such that $supp(\mu_y)=\pi^{-1}(y)$ and
%$supp(\mu_y)\times supp(\mu_y)\subset supp(\lambda)$ (by Lemma \ref{f-p}).
%The distality implies that $supp(\lambda)=R_\pi$. Thus we have that
%It is clear that $\lambda (\Delta)=0$.
%Since for a.e. $y\in X_{n+1}$ we have a measure $\mu_y$ on $X$ such that $supp(\mu_y)\subset \pi^{-1}(y)$
%$$\mu\times_Y\mu(R_\pi)=\int_{y\in Y} \mu_y\times\mu_y(R_\pi)d\nu=1.$$
%As $supp(\mu_y)\subset \pi^{-1}(y)$, a.e. $y\in Y$
%we have $supp(\mu_y\times\mu_y)\subset \pi^{-1}(y)\times \pi^{-1}(y)\subset W$, a.e. $y\in Y$.
%Let $\mu\times_Y\mu=\int_{y\in Y}\mu_y\times\mu_yd\nu$.
%Then $\mu\times_Y\mu$ is an invariant measure on $(X\times X,T\times T)$.
%Moreover, then $supp(\mu\times_Y\mu)\subset W$.

%A unique ergodic minimal distal system $(X,T)$ with $\mu$ is {\it consistent} if $X_n$ is the maximal
%measurable factor with respect to $\mu$. We do not know if Lemma \ref{f-p} holds when $(X,T)$ is consistent.
%If it does then the above proof applies.}

%for $\nu$ a.e. $y\in X_{n+1}$,  $$supp(\mu_y)\times supp(\mu_y)\subset \AP^{[n+1]}.$$
%This implies that
%$R_\pi\subset \AP^{[n+1]}$ by the distality of $(X,T)$.

%\newpage

\section{An example}

In general it is not difficult to find a system whose maximal measurable
and topological factors of order $d$ do not coincide, where $d\le n$.
In fact Lehrer \cite{Lehrer} showed the following result: every ergodic system has a uniquely ergodic and topologically mixing model. Pick any non-periodic ergodic system  with discrete spectrum, and by  Lehrer's result let $(X,T)$ be its uniquely ergodic and topologically mixing model. Since $(X,T)$ is topologically mixing, its maximal equicontinuous factor $Z_1$ is trivial.

By Lemma \ref{mes-top}, for a minimal system of order $n$; the maximal measurable
and topological factors of order $d$ coincide, where $d\le n$. It is natural to ask that for a distal minimal system, if the maximal measurable
and topological factors of order $d$ coincide?

In this section we will construct a strictly ergodic distal system such that $Z_1$ is not
isomorphic to $X_1$. That is, we want to give

\begin{exam}
There is a uniquely ergodic minimal distal system $(X,T)$ with discrete spectrum whose maximal equicontinuous factor
is not equal to $(X,T)$.
\end{exam}

\begin{proof} Let us state the general idea. Let $T_\alpha:\T\lra \T$, $x\mapsto x+\alpha, x\in \T$,
and $m_\T$ be the unique measure of the irrational rotation $T_\alpha$
on $\T$. Then $m_\T$ is the Lebesgue measure on $\T$. We construct $T:\T^2\lra \T^2$ having the form of
$T(x,y)=(x+\alpha,y+u(x))$ such that $(\T^2,T)$ is minimal distal and uniquely ergodic
with the unique measure $\mu=m_{\T^2}=m_\T\times m_\T$, where %$\alpha$ is irrational and
$u:\T\lra \T$ is continuous. At the same time $(\T, T_\alpha)$
is the maximal equicontinuous factor of $(\T^2,T)$.

\medskip
\noindent {\bf Step 1}: The construction of $u$.

\medskip
Let us construct $u$ using some results of \cite{F61}. Choose an irrational $\alpha$ and a subsequence $\{n_k\}$ of integers
with $n_k\not=0, n_{-k}=-n_k$ such that
$$h(\theta)=\sum_{k\not=0}\frac{1}{|k|}(e^{2\pi in_k\alpha}-1)e^{2\pi in_k\theta}$$
and $g(e^{2\pi i\theta})=e^{2\pi i\lambda h(\theta)},$ (where $\lambda\in\R$ will be determined later) are $C^\infty$-functions
of $[0,1)$ and $\T$ respectively. It is clear that
$$h(\theta)=H(\theta+\alpha)-H(\theta),\ \text{where} \ H(\theta)=\sum_{k\not=0}\frac{1}{|k|}e^{2\pi in_k\theta}.$$
Thus, $H(\cdot)\in L^2(0,1)$ is a measurable function. However, $H(\cdot)$ can not correspond to a continuous function since
$\sum_{k\not=0}\frac{1}{|k|}=\infty$ and here the series is not Cesero summable at $\theta=0$, see \cite{Zy}. Therefore, for
some $\lambda$, $e^{2\pi i\lambda H(\theta)}$ can not be a continuous function either.

Considering $R(e^{2\pi i\theta})=e^{2\pi i\lambda H(\theta)}$, we get $R(e^{2\pi i\alpha}s)/R(s)=g(s)$ with $R:\T\lra \T$ measurable
but not continuous.

Put $u(x)=\lambda h(x)+\beta$, where $\alpha,\beta$ are irrational such that $T_{\alpha,\beta}:\T^2\lra \T^2$,
$(x,y)\mapsto (x+\alpha,y+\beta)$ is minimal on $\T^2$ and thus uniquely ergodic.

\medskip
\noindent {\bf Step 2}: The system $(\T^2,T)$ with $T(x,y)=(x+\alpha,y+u(x))$ is strictly ergodic, and $(\T^2,T,\mu)$ is isomorphic to $(\T^2,T_{\alpha,\beta},m_{\T^2})$.

\medskip

It is clear that $m_{\T^2}$ is an invariant measure for $T$.
Define $\phi:\T^2\lra \T^2$, $(x,y)\mapsto (x, y-\lambda H(x))$. It is clear that $\phi$ is measurable and
$m_{\T^2}$ is an invariant measure for $\phi$. Moreover we have the following commuting diagram:
$$
\xymatrix{
  (\T^2,m_{\T^2}) \ar[d]_{\phi} \ar[r]^{T}
                & (\T^2,m_{\T^2}) \ar[d]^{\phi}  \\
  (\T^2,m_{\T^2}) \ar[r]_{T_{\alpha,\beta} }
                & (\T^2,m_{\T^2})             }
$$
since $\phi\circ T(x,y)=(x+\alpha, y+h(x)+\beta-\lambda H(x))=T_{\alpha,\beta}\circ \phi$.
By the fact that $\phi$ is a measurable isomorphism it follows that
$(\T^2, m_{\T^2},T)$ is ergodic (as $(\T^2,m_{\T^2}, T_{\alpha,\beta})$ is ergodic). This implies that
$(\T^2, m_{\T^2},T)$ is uniquely ergodic, \cite[Proposition 3.10]{F}. Since $Supp(m_{\T^2})=\T^2$, it follows that $(\T^2,T)$ is minimal.

%\xymatrix{
%(x,y) \ar[d]_{\phi} \ar[r]^{}
%                & (x+\alpha, y+h(x)+\beta) \ar[d]^{\phi}  \\
%  (x, y-\lambda H(x)) \ar[r]_{}
%                & (x+\alpha, y+h(x)+\beta-\lambda H(x) )             }

\medskip
\noindent {\bf Step 3}: $(\T, T_\alpha)$ is the maximal equicontinuous factor of $(\T^2,T)$.

\medskip
To show this fact we need some preparation.
Let $\pi:(\T^2,T)\lra (\T,T_\alpha)$ be the projection to the first coordinate and $\rho=\int_0^1 u(x)dx$.
We will show that $u$ is an unbounded motion, that is, there exists $x'\in\T$ such that
$$\sup_{n\ge 1}|u(x')+u(x'+\alpha)+\ldots+ u(x'+(n-1)\alpha)-n\rho|=+\infty.$$
This is equivalent to say that
\begin{lem}\label{motion-1}
There exists $x'\in \T$ such that
\begin{equation}\label{equ-51}
\sup_{n\ge 1}|\lambda h(x')+\lambda h(x'+\alpha)+\ldots+ \lambda h(x'+(n-1)\alpha)-n\rho^*|=+\infty,
\end{equation}
where $\rho^*=\int_0^1\lambda h(x)dx=0$.
\end{lem}
The proof of Lemma \ref{motion-1} will be given at the end of the proof.
By Lemma \ref{motion-1} there exists $x'\in\T$ such that $\sup_{n\ge 1}|\sum_{j=0}^{n-1}u(x'+j\alpha)-n\rho|=+\infty.$
WLOG we assume that $\sup_{n\ge 1}\{\sum_{j=0}^{n-1}u(x'+j\alpha)-n\rho\}=+\infty.$
%$$ \sup_{n\ge 1} \{u(x')+u(x'+\alpha)+\ldots+ u(x'+(n-1)\alpha)-n\rho\}=+\infty.$$

We need  another well known lemma, see for example \cite[Lemma 4.1]{Qiao}. Note that the degree of $u$ is zero.

\begin{lem}\label{qiao-1} There exist $x_1,x_2\in\T$ such that
$$\sup_{n\ge 1}\{u(x_1)+u(x_1+\alpha)+\ldots+ u(x_1+(n-1)\alpha)-n\rho\} \le 2$$ and
$$ \inf_{n\ge 1}\{u(x_2)+u(x_2+\alpha)+\ldots+ u(x_2+(n-1)\alpha)-n\rho\} \ge -2.$$
\end{lem}

Now we are ready to show that $(\T,T_\alpha)$ is the maximal equicontinuous factor of $(\T^2,T)$.
Since $\RP^{[1]}(\T^2,T)$ is $T\times T$-invariant and closed it remains to prove that
$$\RP^{[1]}(\T^2,T)\supset \{(x_1,y_1),(x_1,y_2):y_1,y_2\in\T\}.$$

To do this consider
$$\T_{\infty,+}=\{x\in \T: \sup_{n\ge 1}\{u(x)+u(x+\alpha)+\ldots+ u(x+(n-1)\alpha)-n\rho\}=+\infty\}.$$
It is $G_\delta$ and $T_\alpha$-invariant. As $x'\in \T_{\infty,+}$ we know that $x'+i\alpha\in \T_{\infty,+}, i\in\N$.

Fix $y_1,y_2\in\T$ for $\ep>0$ let
$$U_1=(x_1-\ep,x_1+\ep)\times (y_1-\ep,y_1+\ep),\ \text{and}\ U_2=(x_1-\ep,x_1+\ep)\times (y_2-\ep,y_2+\ep).$$

Choose $i_*\in\N$ such  that $x_1^*=x'+i_*\alpha(\text{mod}\ \Z)\in (x_1-\ep,x_1+\ep).$ Since $x_1^*\in \T_{\infty,+}$,
%$\sup_{n\ge 1}u(x_1^*)+u(x_1^*+\alpha)+\ldots+ u(x_1^*+(n-1)\alpha)-n\rho=+\infty$,
there exists $m\in\N$ such that
$$ u(x_1^*)+u(x_1^*+\alpha)+\ldots+ u(x_1^*+(m-1)\alpha)-m\rho\ge 3.$$

Now consider $Q: (x_1-\ep,x_1+\ep)\lra \R$, $x\mapsto \sum_{j=0}^{m-1}u(x+j\alpha)-m\rho+y_1$.
Then by Lemma \ref{qiao-1}
\begin{eqnarray*}
Q(x_1-\ep,x_1+\ep)&\supset& [y_1+ \sum_{j=0}^{m-1}u(x_1+j\alpha)-m\rho, y_1+  \sum_{j=0}^{m-1}u(x_1^*+j\alpha)-m\rho]\\
&\supset& [y_1+2,y_1+3].
\end{eqnarray*}

Thus, there is $x^*\in (x_1-\ep,x_1+\ep)$ such that $Q(x^*)=y_2+\sum_{j=0}^{m-1}u(x_1+j\alpha)-m\rho$ $(\text{mod}\ \Z)$.
Now we have $(x^*,y_1)\in U_1$ and $(x_1,y_2)\in U_2$. Moreover,
\begin{eqnarray*}
T^m(x^*,y_1)&=&\left(x^*+m\alpha, y_1+\sum_{j=0}^{m-1}u(x^*+j\alpha)\right)=(x^*+m\alpha,y_1+Q(x^*)+m\rho)\\
&=&\left(x^*+m\alpha,y_2+\sum_{j=0}^{m-1}u(x_1+j\alpha)\right)\ \ \text{and}\\
T^m(x_1,y_2)&=&\left(x_1+m\alpha,y_2+\sum_{j=0}^{m-1}u(x_1+j\alpha)\right).
\end{eqnarray*}

This implies that $||T^m(x^*,y_1)-T^m(x_1,y_2)||\le ||x^*-x_1||<\ep.$ That is, we have proved that
$((x_1,y_1),(x_1,y_2))\in \RP^{[1]}(\T^2,T)$. It follows that $\RP^{[1]}(\T^2,T)=\{((x,y_1),(x,y_2)):x, y_1,y_2\in\T\}=R_\pi$,
since $\pi$ is distal.
\end{proof}

\medskip
To show Lemma \ref{motion-1}  we need
\begin{lem} \cite[Theorem 3.1]{F61} \label{ff-1} Let $(\Omega_0,T_0)$ be a strictly ergodic system and $\mu_0$ its unique ergodic measure.
Let $\Omega=\Omega_0\times \T$ and let $T:\Omega\lra \Omega$ be defined by $T(w_0,s)=(T_0(w_0),g(w_0)s)$, where $g:\Omega_0\lra \T$
is a continuous function. Then if the equation
$$g^k(w_0)=R(T_0(w_0))/R(w_0)$$
has a solution $R:\Omega_0\lra \T$ which is measurable but not equal almost everywhere to a continuous function, then
$\lim_{N\lra\infty}\frac{1}{N}\sum_{n=0}^{N-1}f\circ T^n(w)$ can not exist
for all continuous functions $f$ and all $w\in \Omega$.
\end{lem}

\noindent {\bf Proof of Lemma \ref{motion-1}}: Let
\begin{equation}\label{equ-52}
\T_\infty=\{x\in\T: \sup_{n\ge 1}|\lambda h(x)+\lambda h(x+\alpha)+\ldots+ \lambda h(x+(n-1)\alpha)-n\rho^*|=+\infty\}.
\end{equation}

It is clear that $\T_\infty$ is a $G_\delta$ and $T_\alpha$-invariant subset, and thus if it is not empty then it is a dense $G_\delta$ subset of $\T$.

Assume the contrary that $\T_\infty=\emptyset$. We {\bf claim} that there exists $M\in\N$ such that
$$|\lambda h(x)+\lambda h(x+\alpha)+\ldots+ \lambda h(x+(n-1)\alpha)-n\rho^*|\le M$$
for any $n\ge 1$ and $x\in\T$. If the claim does not hold, then there exist $x_k\in \T$ and $n_k\lra +\infty$ such that
\begin{equation}\label{equ-53}
|\lambda h(x_k)+\lambda h(x_k+\alpha)+\ldots+ \lambda h(x_k+(n_k-1)\alpha)-n_k\rho^*|>k.
\end{equation}

Consider $$\T_l=\{x\in\T:\exists n\in\N\ \text{s.t.}\ |\lambda h(x)+\lambda h(x+\alpha)+\ldots+ \lambda h(x+(n-1)\alpha)-n\rho^*|>l\},$$
$l\in\N$. It is clear that $\T_l$ is open and $\T_\infty=\cap_{l\in\N}\T_l$.
Now we show that $\T_l$ is a dense open subset for any $l\in\N$.

Fix $l\in\N$. For any non-empty open subset $V$ of $\T$, there exists $r=r(V)\in\N$
such that $\cup_{i=0}^r T_\alpha^{-i}V=\T$. Choose $k>l+r|\rho^*|+r\max_{x\in\T}|\lambda h(x)|$. By (\ref{equ-53}) we have
$$|\lambda h(x_k+i\alpha)+ \lambda h(x_k+i\alpha+\alpha)+\ldots+\lambda h(x_k+i\alpha+(n_k-i-1)\alpha)-(n_k-i)\rho^*|>l$$
for $i=0,1,\ldots,r.$ That is, $x_k+i\alpha\in\T_l$. Since $\cup_{i=0}^r T_\alpha^{-i}V=\T$,
there exists $0\le i\le r$ with $x_k +i\alpha\in V\cap \T$. This implies that $\T_l$ is dense, and hence
$\T_\infty$ is dense, a contradiction. This proves the {\bf claim}.

\medskip
Consider now $S:\T\times \R\lra \T\times \R$, $(x,y)\mapsto (x+\alpha,y+\lambda h(x)-\rho^*)$. Since
$$S^n(x,y)=(x+n\alpha, y+\lambda h(x)+\ldots+\lambda h(x+(n-1)\alpha)-n\rho^*)$$
for any $n\ge 0$, we have $\{S^n(0,0):n\ge 0\}\in \T\times [-M,M]$. Thus,
$E=\overline{\{S^n(0,0):n\ge 0\}} \subset \T\times [-M,M]$
is an $S$-invariant compact subset. This deduces that there is a minimal subset $F\subset E$.

As $S$ is distal, $(F,S)$ is a minimal distal system and $p:\T\times \R\lra \T$, $(x,y)\mapsto x$ is a factor map
from $(F,S)$ to $(\T,T_\alpha)$.

Let $I(x)=\{y\in \R:(x,y)\in F\}$ for any $x\in\T$. Fix $x\in \T$ we {\bf claim} that $|I(x)|=1$. In fact let $y_1^*=\max I(x)$ and $y_2^*=\min I(x)$,
then $y_2^*\le y_1^*$. As $(F,S)$ is minimal, there are $\{n_k\}$ such that $S^{n_k}(x,y_2^*)\lra (x,y_1^*)$.
This implies
$$y_2^*+\lambda h(x)+\ldots+\lambda h(x+(n_k-1)\alpha)-n_k\rho^*\lra y_1^*$$ and
we assume that
$$y_1^*+\lambda h(x)+\ldots+\lambda h(x+(n_k-1)\alpha)-n_k\rho^*\lra y_3^*\in I(x).$$
Thus, $y_1^*\le y_3^*$ and hence $y_1^*= y_3^*$. This implies $y_1^*= y_2^*$, i.e. $|I(x)|=1$. This ends the proof of the {\bf claim}.

\medskip
By what we just proved we know that there exists $\tilde{g}:\T\lra \R$ continuous such that $\{(x,g(x)):x\in \T\}=F$.
Note that the continuity of $g$ follows from the fact that the projection $p:E\lra \T$ is one to one.

Since $SF=F$ we get that $$(x+\alpha,\tilde{g}(x+\alpha))=(x+\alpha,\tilde{g}(x)+\lambda h(x)-\rho^*)$$ for any $x\in \T$.
As $\rho^*=0$ we know that
$$\lambda h(x)=\tilde{g}(x+\alpha)-\tilde{g}(x),\ \  \forall x\in\T.$$

Now consider $U:\T^2\lra \T^2$, $(w_1,w_2)\mapsto (w_1e^{2\pi i \alpha}, g(w_1)w_2)$, where
$g(e^{2\pi i\theta})=e^{2\pi i\lambda h(\theta)}.$ Since $R(e^{2\pi i\alpha}s)/R(s)=g(s)$ with $R:\T\lra \T$ measurable
but not continuous. By Lemma  \ref{ff-1} there exists $f\in C(\T^2)$ and $(w_1,w_2)\in\T^2$ such that
$$\lim_{N\lra\infty}\frac{1}{N}\sum_{n=0}^{N-1}f\circ U^n(w_1,w_2)\ \text{does not exist}.$$

At the other hand since $\lambda h(x)=\tilde{g}(x+\alpha)-\tilde{g}(x), \forall x\in\T$, by writing $w_1=e^{2\pi i x_1}$ and $w_2=e^{2\pi i y_1}$ we have
\begin{eqnarray*}
\frac{1}{N}\sum_{n=0}^{N-1}f\circ U^n(w_1,w_2)&=&\frac{1}{N}\sum_{n=0}^{N-1}f(e^{2\pi i (x_1+n\alpha)},e^{2\pi i (y_1+\tilde{g}(x_1+n\alpha)-\tilde{g}(x_1))})\\
&=&\frac{1}{N}\sum_{n=0}^{N-1}\tilde{H}(n\alpha)=\frac{1}{N}\sum_{n=0}^{N-1}\tilde{H} (T_\alpha^n(0))\lra \int_0^1 \tilde{H}(t)dt,
\end{eqnarray*}
by the unique ergodicity of $(\T,T_\alpha)$, where $\tilde{H}(t)=f(e^{2\pi i (x_1+t)},e^{2\pi i (y_1+\tilde{g}(x_1+t)-\tilde{g}(x_1))})$
is a periodic continuous function of period $1$ for $t$.
It is a contradiction. Thus, we have proved that $\T_\infty\not=\emptyset$, and this ends the proof.  \hfill $\square$

%\medskip


\begin{thebibliography}{SSS}



%\bibitem{A98} I, Assani, {\it Multiple recurrence and almost sure convergence for
%weakly mixing dynamical systems}, Israel J. Math., {\bf 103} (1998),
%111-124.

\bibitem{Au88} J. Auslander, \textit{Minimal flows and their
extensions}, North-Holland Mathematics Studies {\bf 153} (1988),
North-Holland, Amsterdam.

\bibitem{BL} V. Bergelson and A. Leibman, \textit{Polynomial extensions of van der Waerden's and Szemeredi's theorems}, Journal of AMS, {\bf 9} (1996), no.3, 725-753


%\bibitem{BM} V. Bergelson and R. McCutcheon, {\it An ergodic IP polynomial Szemeredi theorem}, Mem. Amer. Math. Soc. 146 (2000), viii+106pp.

\bibitem{B1} F. Blanchard, \textit{Fully positive topological
entropy and topological mixing}, Symbolic dynamics and its applications, AMS
Contemporary Mathematics, \textbf{135}(1992), 95-105.

\bibitem{B2} F. Blanchard, \textit{A disjointness theorem
involving topological entropy}, Bull. de la Soc. Math. de France,
\textbf{121}(1993), 465-478.


\bibitem{BHR} F. Blanchard, B. Host and S. Ruette, \textit{Asymptotic
pairs in positive-entropy systems}, Ergod. Th. and Dynam. Sys., {\bf
22} (2002), No. 3, 671--686.

\bibitem{Br} I.U. Bronstein, \textit{Extensions of minimal transformation
groups}, Martinus Nijhoff Publications, The Hague, 1979.

\bibitem{CS}  F. Cai and S. Shao, \textit{Topological characteristic factors along cubes of minimal systems},  Discrete Contin. Dyn. Syst. 39 (2019), no. 9, 5301--5317.


\bibitem{D-Y} P. Dong, S. Donoso, A. Maass, S. Shao, X. Ye, \textit{Infinite-step nilsystems, independence andcomplexity},  Ergod. Th. and Dynam. Sys., {\bf 33}(2013), 118--143.

\bibitem{Ellis} R. Ellis, \textit{Lectures on topological dynamics}, W. A.
Benjamin, Inc., New York 1969

\bibitem{EG} R. Ellis and W. Gottschalk, \textit{ Homomorphisms of transformation
groups}, Trans. Amer. Math. Soc., {\bf 94} (1960), 258-271.

\bibitem{EGS} R. Ellis, S. Glasner and L. Shapiro, \textit{Proximal-Isometric Flows}, Adv. Math. {\bfseries 17}, (1975), 213--260.

\bibitem{EK} R. Ellis and H. Keynes, \textit{A characterization of the equicontinuous
structure relation}, Trans. Amer. Math. Soc., {\bf 161} (1971),
171--181.

\bibitem{F61} H. Furstenberg, \textit{Strict ergodicity and transformation of the torus},
Amer. J. Math., {\bf 83} (1961), 573--601.

\bibitem{F63} H. Furstenberg,\textit{The structure of distal flows}, Amer. J.
Math. 85, 1963, 477--515.

\bibitem{F77} H. Furstenberg, \textit{Ergodic behavior of diagonal measures
and  a theorem of Szem\'{e}redi on arithmetric progressions}, J. Analyse Math. {\bfseries 31}, (1977), 204--256.

\bibitem{F} H. Furstenberg, \textit{Recurrence in ergodic theory and
combinatorial number theory}, M. B. Porter Lectures. Princeton
University Press, Princeton, N.J., 1981.

\bibitem{G75} S. Glasner, \textit{Relatively invariant measures}, Pacific J. Math. 58 (1975), no.2, 393--410.

\bibitem{G93} E. Glasner, \textit{Minimal nil-transformations of class two}, Israel
J. Math., {\bf 81}(1993), 31--51.

\bibitem{G94} E. Glasner, \textit{Topological ergodic decompositions and
applications to products of powers of a minimal transformation}, J.
Anal. Math., {\bf 64} (1994), 241-262.

\bibitem{GGY-16} E. Glasner, Y. Gutman and X. Ye, \textit{Higher order regionally proximal equivalence
relations for general minimal group actions}, Adv. Math. {\bf 333} (2018), 1004-1041.


\bibitem{GW} E. Glasner and B. Weiss, \textit{Quasi-factors of zero-entropy systems},
J. Amer. Math. Soc., {\bfseries 8}, (1995), 665--686.

\bibitem{HK05} B. Host and B. Kra, \textit{Nonconventional averages and
nilmanifolds}, Ann. of Math., {\bfseries 161} (2005) 398--488.

\bibitem{HK18} B. Host and B. Kra, \textit{Nilpotent Structures in Ergodic Theory},
Mathematical Surveys and Monographs {\bf 236}, AMS, 2018.%https://sites.math.northwestern.edu/~kra/papers/book.html

\bibitem{HKM} B. Host, B. Kra and A. Maass, \textit{Nilsequences and a structure
theory for topological dynamical systems}, Adv. Math.,
{\bf 224} (2010) 103--129.

\bibitem{HKM-1} B. Host, B. Kra and A. Maass, \textit{Variations on topological recurrence},
Monatsh. Math. {\bf 179} (2016),  no. 1, 57-89.

%\bibitem{HM} B. Host and A. Maass, \textit{ Nilsyst\`emes d'ordre deux et
%parall\'el\'epip\`edes}, Bull. Soc. Math. France, {\bf 135} (2007) 367--405.

\bibitem{HLY} W. Huang, H. Li and X. Ye, \textit{Family-independence for topological and measurable
dynamics}, Trans. Amer. Math. Soc., {\bf 364}(2012), 5209-5242.

\bibitem{HLY1} W. Huang, H. Li and X. Ye, \textit{Localization and dynamical Ramsey
property}, Preprint.

\bibitem{HSY1} W. Huang, S. Shao and X. Ye, \textit{Nil Bohr-sets and almost automorphy of higher
order}, Mem. Amer. Math. Soc. 241 (2016), no. 1143, v+83 pp.

\bibitem{HSY-16} W. Huang, S. Shao and X. Ye, \textit{ Topological correspondence of multiple
ergodic averages of nilpotent group actions}, . J. Anal. Math. 138 (2019), no. 2, 687--715.

\bibitem {HY02} W. Huang  and X. Ye, \textit{An explicit scattering, non-weakly mixing example and weak disjointness}, Nonlinearity, {\bf 15}(2002), 1--14.

\bibitem{HY06} W. Huang and X. Ye, \textit{A local variational relation and applications}, Israel J. Math., {\bfseries 151}, (2006), 237--280.

%\bibitem {HY1} W. Huang  and X. Ye, {\it Topological
%complexity, return times and weak disjointness}, Erg. Th. Dynam.
%Sys., {\bf 24} (2004), 825-846.


\bibitem{KL} D. Kerr and H. Li, \textit{Independence in topological and
$C^*$-dynamics}, Math. Ann., {\bfseries 338}, (2007), 869--926.


\bibitem {KN} H. Keynes and D. Newton, \textit{Real prime flows}, Trans. Amer. Math.
Soc., {\bf 217}(1976), 237--255.

\bibitem{KO12} D.~Kwietniak and P.~Oprocha, \textit{On weak mixing, minimality and weak disjointness of all iterates}, Ergod. Th. and Dynam. Sys., {\bf 32}(2012), 1661--1672.

\bibitem{Lehrer} E. Lehrer, \textit{Topological mixing and uniquely ergodic
systems}, Israel J. Math. 57 (1987), no. 2, 239--255.


\bibitem{Mc} D. C. McMahon, \textit{Relativized weak disjointness and relatively
invariant measures}, Trans. Amer. Math. Soc., {\bf 236} (1978),
225--237.

%\bibitem{M} T.K.S. Moothathu, \textit{Diagonal points having dense orbit}, Colloq. Math., {\bf 120} (2010), 127--138.

\bibitem{OW} D. Ornstein and B. Weiss, \textit{Mean distality and
tightness}, Tr. Mat. Inst. Steklova 244 (2004), Din. Sist. i Smezhnye Vopr. Geom., 312--319; translation in Proc. Steklov Inst. Math. 2004, no. 1(244), 295--302

\bibitem{Qiao} Y. Qiao, \textit{Topological complexity, minimality and the system of order two on torus}, Sci. China Math. {\bf 59}(2016), 503--514.

\bibitem{SY} S. Shao and X. Ye, \textit{Regionally proximal relation of order $d$ is an equivalence one for minimal systems and a combinatorial consequence},  Adv. Math., {\bf 231} (2012), 1786--1817.

\bibitem{Vr} J. de Vries, \textit{Elements of Topological Dynamics}, Kluwer Academic Publishers ({\bf 993}), Dordrecht.


\bibitem{Weiss} B. Weiss, \textit{Multiple recurrence and doubly minimal systems}, Contemp. Math., 215,
189-196,  Amer. Math. Soc., Providence, RI, 1998.

\bibitem{V68} W. A. Veech, \textit{The equicontinuous structure relation for
minimal Abelian transformation groups}, Amer. J. Math., {\bf
90}(1968), 723--732.

\bibitem{V77} W. A. Veech, \textit{Topological systems}, Bull. Amer. Math. Soc.,
83(1977), 775-830.

\bibitem{Z} T. Ziegler, \textit{Universal characteristic factors and Furstenberg averages}. J. Amer. Math. Soc., {\bf 20} (2007), 53--97.

\bibitem{Zy} A. Zygmund, \textit{Trigonormetric series}, 2nd edition, University press, Cambridge, 1959.

\end{thebibliography}
\end{document}